\documentclass{article}
\pdfoutput=1
\usepackage[utf8]{inputenc}
\usepackage{enumerate}
\usepackage{amsfonts,amsmath,amssymb,amsthm,bbm,bm}
\usepackage{mathrsfs}
\usepackage{mathtools}
\usepackage{xcolor}
\usepackage{soul}
\usepackage{forest}
\setstcolor{red}

\usepackage{graphicx}
\usepackage{caption}
\usepackage{subcaption}
\usepackage{float}

\usepackage{tikz}
\usepackage{pgfplots}

\usepackage[style=alphabetic,backend=biber, sorting = nyt, maxbibnames=99]{biblatex}
\usepackage{sepfootnotes}
\usepackage[hyperfootnotes=false]{hyperref}
\hypersetup{colorlinks}
\usepackage[colorinlistoftodos]{todonotes}

\usepackage{thmtools}
\usepackage{hyperref}
\usepackage[capitalise]{cleveref}

\newtheorem{theorem}{Theorem}[section]
\newtheorem{corollary}[theorem]{Corollary}

\newtheorem{assumption}[theorem]{Assumption}
\newtheorem{lemma}[theorem]{Lemma}

\theoremstyle{remark}
\newtheorem{remark}[theorem]{Remark}

\theoremstyle{definition}
\newtheorem{definition}[theorem]{Definition}
\newtheorem{example}[theorem]{Example}

\newcommand{\E}{\mathbb{E}}
\newcommand{\R}{\mathbb{R}}
\newcommand{\N}{\mathbb{N}}

\newcommand{\sP}{\mathbb{P}}
\newcommand{\calF}{\mathcal{F}}
\newcommand{\calP}{\mathcal{P}}

\newcommand{\calC}{\mathcal{C}}

\def\calB{{\mathcal{B}}}

\def\calX{{\mathcal{X}}}
\def\calK{{\mathcal{K}}}

\newcommand{\W}{\mathcal{W}}
\newcommand{\AW}{\mathcal{AW}}

\newcommand{\TV}{\mathrm{TV}}
\newcommand{\ATV}{\mathrm{ATV}}

\newcommand{\cpl}{\mathrm{Cpl}}
\newcommand{\bccpl}{\mathrm{Cpl}_{\mathrm{bc}}}

\numberwithin{equation}{section}

\title{Estimating causal distances with non-causal ones}

\author{Beatrice Acciaio, Songyan Hou, Gudmund Pammer}
\date{\today}

\begin{document}
\maketitle

\begin{abstract}
The adapted Wasserstein ($\AW$) distance refines the classical Wasserstein ($\W$) distance by incorporating the temporal structure of stochastic processes. This makes the $\AW$-distance well-suited as a robust distance for many dynamic stochastic optimization problems where the classical $\W$-distance fails. However, estimating the $\AW$-distance is a notably challenging task, compared to the classical $\W$-distance. In the present work, we build a sharp estimate for the $\AW$-distance in terms of the $\W$-distance, for smooth measures. This reduces estimating the $\AW$-distance to estimating the $\W$-distance, where many well-established classical results can be leveraged. As an application, we prove a fast convergence rate of the kernel-based empirical estimator under the $\AW$-distance, which approaches the Monte-Carlo rate ($n^{-1/2}$) in the regime of highly regular densities. These results are accomplished by deriving a sharp bi-Lipschitz estimate of the adapted total variation distance by the classical total variation distance.
\\


\end{abstract}

\section{Introduction}
\label{sec:intro}
Dynamic optimization problems on stochastic processes arise naturally in management sciences, economics, and finance, to support decision-making under uncertainty. In this context, estimating the change of optimal values with respect to small changes in the law of the underlying stochastic process is crucial for evaluating robustness and quantifying uncertainty.
Wasserstein ($\mathcal W$) distances are ubiquitously used to measure proximity of probability distributions, but they fail to provide continuity for dynamic optimization problems. This is because $\W$-distances do not take into account the temporal flow of information inherent in stochastic processes.
To address this issue, adapted Wasserstein  ($\AW$) distances, also known as nested Wasserstein distances, have been introduced as a variant of $\W$-distances that accounts for the time-causal structure of stochastic processes 
and the underlying information flow (\cite{ruschendorf1985wasserstein}).
$\AW$-distances serve as robust distances for many dynamic stochastic optimization problems across several fields, particularly in mathematical finance and economics; see \cite{backhoff2020adapted, bion2009time, glanzer2019incorporating, bartl2023sensitivity}. More precisely, we have estimates of the type
\[
|V(\mu) - V(\nu)| \leq C \, \AW(\mu,\nu),
\]
where \(\mu\) and \(\nu\) are distributions on a path space, \(V(\cdot)\) denotes the optimal value of a multi-period stochastic control problem under the corresponding measure, and \(C > 0\) is a constant depending only on the structure of the control problem (e.g., boundedness of admissible strategies, Lipschitz continuity of utility functions), but independent of \(\mu\) and \(\nu\); see \cite{pflug2014multistage,backhoff2020adapted}. While the nested structure of the $\AW$-distances ensures robustness, it poses greater challenges on the estimation of such distances compared to the classical $\W$-distances:
\begin{enumerate}[-]
    \item \emph{Theoretical challenge}: Explicit closed-form expressions for $\AW$-distances are scarce and available only in specific cases, such as SDEs (see \cite{backhoff2020adapted, bion2019wasserstein, cont2024causal, robinson2024bicausal}) and discrete Gaussian processes (\cite{gunasingam2025adapted,acciaio2024entropic}).
    \item \emph{Statistical challenge}: $\AW$-distances are so strong that even empirical measures do not converge to the true measure under them. Alternative consistent empirical estimators are required, such as adapted empirical measures (see \cite{backhoff2022estimating,acciaio2024convergence}) or smoothed empirical measures (see \cite{pflug2016empirical, hou2024convergence}). 
    \item \emph{Computational challenge}: In general, the computation of $\AW$-distances involves backward dynamic programming (see \cite{pichler2022nested,eckstein2024computational,bayraktar2023fitted}), which is computationally very costly when the path length is long.
\end{enumerate}
Although $\AW$-distances are generally difficult to estimate, recent results show that for probability measures with certain regularity properties, they can be bounded from above by classical Wasserstein distances, such as $\AW_1(\mu,\nu) \lesssim \W_1^{\frac{1}{2}}(\mu,\nu)$;\footnote{We write $\lesssim$ to denote that the inequality holds up to a positive multiplicative constant.} see \cite{blanchet2024bounding}. Motivated by this, the aim of the present paper is to establish a sharp estimate of $\AW$-distances from above by $\W$-distances for regular measures, where we quantify the regularity of a measure by the Sobolev norm of its (Lebesgue) density. For measures $\mu$ and $\nu$ with Sobolev densities of order $k$, we show that $\AW_1(\mu,\nu)\lesssim \W_1^{\frac{k}{k+1}}(\mu,\nu)$ and that the order $\frac{k}{k+1}$ is sharp. As an application, we demonstrate how this sharp estimate simplifies the task of estimating $\AW$-distances to the estimation of $\W$-distances, for which many well-established tools are available. Specifically, we prove a fast convergence rate for the empirical estimator under $\AW_1$, meaning  that the convergence rate approaches to the Monte-Carlo rate ($n^{-1/2}$) as $k$ tends to infinity. The proof follows directly by leveraging known convergence rates under Wasserstein distances for smooth measures, whereas deriving such a result would be considerably more involved without this estimate.

\paragraph*{Main Contributions}
In the present paper, we derive upper bounds for adapted Wasserstein distances in terms of total variation and classical Wasserstein distances for smooth measures. Specifically, for any probability measures $\mu,\nu$, we establish the following chain of estimates (see Theorem~\ref{thm:weighted_main_inequality} and Corollary~\ref{thm:AWpTVp}):
\begin{equation}    \label{eq:intro.AW.W}
    \AW_1(\mu,\nu) \lesssim \ATV(\mu,\nu) \lesssim \TV(\mu,\nu),
\end{equation}
where $\TV$ denotes the total variation distance and $\ATV$ its adapted counterpart; see explicit expressions in Theorem~\ref{thm:weighted_main_inequality} and Corollary~\ref{thm:AWpTVp}.
The estimate \eqref{eq:intro.AW.W} was first noticed in \cite{eckstein2024computational}---where $\ATV$ was first introduced---albeit under a compactness assumption and with a constant that scales exponentially in the number of time steps. This estimate was further generalized to the non-compact setting in \cite{hou2024convergence}, where weighted versions of $\ATV$ and $\TV$ were introduced to establish \eqref{eq:intro.AW.W}. In this work, we obtain sharp constants in $\ATV(\mu,\nu) \lesssim \TV(\mu,\nu)$ for both weighted and classical versions, and, surprisingly, the sharp constant $C = 2T - 1$ for the classical $\ATV$ and $\TV$. Remarkably, this constant does not scale exponentially as in \cite{eckstein2024computational}, but linearly in the path length $T$, thereby taming the curse of dimensionality.
If $\mu,\nu$ have Sobolev regularity of order $k$, for some $k\in\N$, we further have the following estimate (see Theorem~\ref{thm:AWtoW}):
\begin{equation}
\label{eq:intro:AW2W}
    \AW_p^p(\mu,\nu)\lesssim \left(\|f \|_{k,1} + \|g \|_{k,1}\right)^{\frac{1}{k+1}} \W_1^{\frac{k}{k+1}}(\mu,\nu),
\end{equation}
for $p\geq 1$. In particular, for $\W_1$ we demonstrate that the order of the exponent $\frac{k}{k+1}$ in the above estimate is optimal for measures with Sobolev regularity of order $k$; see Example~\ref{ex:orderissharp}.

Having estimated adapted distances with non-adapted ones, we can now rely on the rich theory that has already been developed for Wasserstein distances and total variation distances in the last decades. For example, nonparametric density estimation problems with error measured by the Wasserstein distance are popular in many areas of statistics and machine learning. In particular, minimax-optimal rates have been established for smooth densities under Wasserstein distances in \cite{niles2022minimax, divol2022measure, manole2024plugin}. We can then leverage such results and combine them with our estimates to build an empirical kernel density estimator $f^{(n)}_{k}$ based on $n$ observations of the underlying measure $\mu$. Specifically, in the setting of Theorem~\ref{thm:fast_rate_compact}, where the underlying measure  has density $f$, bounded away from zero and with finite $s$-order Besov norm (see Definition~\ref{def:besov}) with $s > k \in \N$, we show that the estimator has the following moment convergence rate:
\begin{equation*}
\E\big[\AW_1(\mu,\hat \mu_n)\big] \lesssim n^{-\frac{k}{2k + dT}\frac{k}{k+1}},
\end{equation*}

where $\hat \mu_n$ is the measure with density $\hat f_n$; see Theorem~\ref{thm:fast_rate_compact}. The above rate of convergence approaches the Monte-Carlo rate as the smoothness $k$ of the underlying density tends to infinity.
\paragraph*{Related Literature}

Bounding relationships between different probabilistic distances gives powerful tools in applications; see \cite{gibbs2002choosing} for a summary of such results among several important probability metrics. For the $\AW$-distance, first established in \cite{eckstein2024computational} that $\W \leq \AW \lesssim \ATV \lesssim \TV$ in the compact setting. The bounds $\W \leq \AW \lesssim \ATV$ follow directly from definitions, whereas the inequality $\ATV \lesssim \TV$ is surprising, as it shows that an adapted distance can be bounded by a non-adapted one. This estimate was further generalized to a non-compact setting in \cite{hou2024convergence}. In \cite{blanchet2024bounding}, Blanchet, Larsson, Park, and Wiesel leverage this estimate to derive a bound of the $\AW$-distance in terms of the classical $\W$-distance:
\begin{equation}
\label{eq:intro:AW2W_their}
\AW_p^p(\mu,\nu) \lesssim \W_1(\mu,\nu)^{\frac{p}{p+1}},
\end{equation}
where $p \geq 1$, $K \subseteq \R^{dT}$ is compact, and $\mu, \nu \in \calP(K)$ are probability measures with Lipschitz kernels under the $\W$-distance; see Corollary 31(b) (with $\alpha = 1$) in \cite{blanchet2024bounding}. Notably, inequality \eqref{eq:intro:AW2W_their} is closely related to our main result, though the two bounds are derived through completely different approaches. In \cite{blanchet2024bounding}, the authors first bound the smoothed adapted Wasserstein distance $\AW_p^{(\sigma)}(\mu,\nu)$ by $\W_1(\mu,\nu)$, and then estimate the difference between $\AW_p^{(\sigma)}(\mu,\nu)$ and $\AW_p(\mu,\nu)$ when $\mu, \nu$ have Lipschitz kernels. 
In contrast, under the same compactness assumption, our approach directly bounds $\AW_p(\mu,\nu)$ by the total variation $\TV(\mu,\nu)$ (see Definition~\ref{def:wTV_wATV}), and then bounds $\TV(\mu,\nu)$ by $\W_1(\mu,\nu)$ under the additional assumption that $\mu, \nu$ have smooth densities. Since the two results rely on different regularity assumptions---kernel regularity in \cite{blanchet2024bounding} versus density smoothness in our work---they are difficult to compare directly. To the best of the authors' knowledge, neither assumption is stronger than the other. Nevertheless, in terms of convergence order, when $k > p$, our bound \eqref{eq:intro:AW2W} dominates \eqref{eq:intro:AW2W_their}, and vice versa when $p > k$. Also, an intermediate step in our approach bounding $\TV_p$ by $\W_1$ generalizes the result in \cite{chae2020wasserstein}, where Chae and Walker established a sharp bound for the total variation in terms of the Wasserstein distance for smooth measures.

Bounding techniques such as these are widely used to study convergence rates of empirical estimators, especially when direct computation of the probabilistic distance is intractable. In the seminal work \cite{dereich2013constructive}, Dereich, Scheutzow, and Schottstedt introduced a partition-based distance to bound the Wasserstein distance. This technique was subsequently used in the influential paper \cite{fournier2015rate}, where Fournier and Guillin established sharp convergence rates of empirical measures under $\W$-distances. Similarly, Niles-Weed and Berthet in \cite{niles2022minimax} proved minimax convergence rates for smooth densities under $\W$-distances by bounding them with Besov norms of negative smoothness; related bounds also appear in \cite{peyre2018comparison} and \cite{shirdhonkar2008approximate}. Compared to classical $\W$-distances, empirical convergence rates under $\AW$-distances are more challenging to obtain, since empirical measures may not converge to the underlying measure. Two alternative estimators have been studied: the smoothed empirical measure and the adapted empirical measure. The adapted empirical measure was introduced by Backhoff et al.\ in \cite{backhoff2022estimating}, where sharp convergence rates were provided in the compact setting, result later extended to general measures in \cite{acciaio2024convergence}. For the smoothed empirical measure, Pichler and Pflug first established convergence under $\AW$ in \cite{pflug2016empirical}, and convergence rates were later derived in \cite{hou2024convergence}. In the present paper, we propose a third empirical estimator $\mu^{(n)}_{k}$ of the underlying measure $\mu$ based on $k$-th order smooth kernels (see Definition~\ref{def:kthorderkernel}) and, as an application of our bound \eqref{eq:intro:AW2W}, we establish fast convergence rate of $\E[\AW_1(\mu, \mu^{(n)}_{k})]$ under the $\AW$-distance for smooth measures. 
A similar fast convergence rate has also been established in \cite{hou2024convergence} (see Theorem~4.5), which states that for compact measure $\mu$ and $p\geq 1$, $\E[\AW_{p}^{(\sigma)}(\mu, \mu^n)] \lesssim n^{-1/2}$, where $\mu^n$ is the empirical measure of $\mu$. This is later generalized to $\mu \in \calP_p(\R^{dT})$ in \cite{larsson2025fast} (see Theorem~4). Let us point out a structural difference between the results in \cite{hou2024convergence, larsson2025fast} and the fast rate in the present paper. The results in \cite{hou2024convergence, larsson2025fast} establish convergence rates for smoothed adapted Wasserstein distance with a fixed $\sigma > 0$, while the fast rate in the present paper holds for adapted Wasserstein distance without smoothing the underlying measure under Sobolev regularity.

\paragraph*{Organization of the paper}
We conclude Section~\ref{sec:intro} by introducing notations used throughout the paper.
In Section~\ref{sec:main}, we introduce the setting and state our main results. In Section~\ref{sec:ATV_TV}, we prove the sharp estimate of $\ATV$-distances by $\TV$-distances. In Section~\ref{sec:AW_W}, we prove the sharp estimate of $\AW$-distances by $\W$-distances for smooth measures. In Section~\ref{sec:fast}, we prove the fast convergence rate of our kernel density estimator under $\AW$-distances.

\paragraph*{Notations}
We work on the path space $\R^{dT}$, where $\R^d$ is the state space and $T \in \N$ is the number of time steps. We write $X=(X_t)_{t = 1}^T$ for the canonical process on the path space, and $x = (x_1,\dots,x_T)\in\R^{dT}$ for a generic path. For $x = (x_t)_{t = 1}^T \in\R^{dT}$ and natural numbers $1\leq s\leq t\leq T$, we use the shortening $x_{s:t}=(x_s,\ldots,x_t)$. For $\mu \in \calP(\R^{dT})$, we denote the up-to-time-$t$ marginal of $\mu$ by $\mu_{1:t}$ (with  $\mu_1 = \mu_{1:1}$ for simplicity), and the kernel of $\mu$ w.r.t.\ $x_{1:t}$ by $\mu_{x_{1:t}}$, so that $\mu(dx_{t+1}) = \int_{\R^{dt}}\mu_{x_{1:t}}(dx_{t+1})\mu_{1:t}(dx_{1:t})$.
We denote by $\calP(\R^{dT})$ the set of all probability measures on $\R^{dT}$ and, for $p\geq 1$, we write $\calP_p(\R^{dT})$ for the subset of measures with finite $p$-th moment.
Given $n,m\in\N$, a measure $\mu\in\calP(\R^n)$ and a measurable map $f : \R^n \to \R^m$, we denote by $f_\# \mu$ the push-forward measure of $\mu$ under $f$, that is, the measure $\nu\in\calP(\R^m)$ satisfying $\nu(A)=\mu(f^{-1}(A))$, for any Borel set $A\in\mathcal B(\R^m)$.
The set of finite, signed Borel measures on $\R^n$, denoted by $\mathcal M(\R^n)$, endowed with the order relation $\mu \le \nu$ if $\mu(A) \le \nu(A)$ for all $A \in \mathcal B(\R^n)$, is a complete lattice and we write $\mu \wedge \nu$ resp.\ $\mu \vee \nu$ for the minimum resp.\ maximum of $\mu$ and $\nu$ w.r.t.\ this order.
For $n\in\N$, $p\geq 1$, $x\in\R^n$, we  denote
$|x|^p=\sum_{k=1}^n |x_k|^p$.
For a multi-index $\alpha = (\alpha_1, \alpha_2, \dots, \alpha_d) \in \mathbb{N}_0^d$, we denote its length  by $|\alpha| := \alpha_1 + \alpha_2 + \cdots + \alpha_d$ and set $x^\alpha \coloneqq x_1^{\alpha_1} x_2^{\alpha_2} \cdots x_d^{\alpha_d}$, $x\in \R^d$.

\section{Main Results}
\label{sec:main}
We start this section by introducing all the relevant distances. Next, we establish the chain of bounds \( \AW \lesssim \ATV \lesssim \TV \), and subsequently show the additional bound \(  \TV \lesssim \W \). Finally, we leverage this full chain of inequalities to prove the fast convergence rate of the empirical density estimator under $\AW_1$-distance.

\subsection{Adapted Wasserstein and TV distances}
\label{subsec:pre}
First we recall the definition of Wasserstein distance. 
\begin{definition}
    For $p\geq 1$, the $p$-th order \textit{Wasserstein distance} $\W_{p}$ on $\calP_p(\R^{dT})$ is defined by 
    \begin{equation*}
        \W_{p}^p(\mu,\nu) = \inf_{\pi \in \cpl(\mu,\nu)}\int_{\R^{dT}\times\R^{dT}} |x-y|_p^p \,\pi(dx,dy),
    \end{equation*}
    where $\cpl(\mu,\nu)$ denotes the set of couplings between $\mu$ and $\nu$, that is, probabilities on $\R^{dT}\times\R^{dT}$ with first marginal $\mu$ and second marginal $\nu$.
\end{definition}
In the present work, we are interested in the case of couplings satisfying a bicausality constraint. 
This is motivated by the fact that the optimal value of many stochastic optimization problems in dynamic settings---such as optimal stopping problems and utility maximization problems---is not continuous with respect to the Wasserstein distance. This means that stochastic models can be arbitrarily close to each other under Wasserstein distance, whilst the values of the mentioned optimization problems under those models can be quite far from each other. This discrepancy arises because the optimal coupling under the Wasserstein distance does not account for adaptedness with respect to the information flow, which is crucial for admissible optimizers of such problems in dynamic settings. Therefore, we focus on a smaller set of couplings, for which adaptedness with respect to the information flow is respected. Crucially, the mentioned optimization problems are Lipschitz continuous with respect to the adapted Wasserstein distances \cite{backhoff2020adapted}.
For this, denote by $(X,Y)=((X_t)_{t = 1}^T,(Y_t)_{t = 1}^T)$ the canonical process on $\R^{dT}\times\R^{dT}$. Then, a coupling $\pi\in\cpl(\mu,\nu)$ is called bicausal if, for all $t=1,\ldots T-1$, 
\[
\pi(Y_t|X)=\pi(Y_t|X_1,\ldots,X_t)\quad \text{and}\quad \pi(X_t|Y)=\pi(X_t|Y_1,\ldots,Y_t),
\]
in which case we write $\pi\in\bccpl(\mu,\nu)$.\footnote{In the optimal transport literature, causal and bicausal couplings have been introduced in a more general setting, with respect to general filtrations on the path space; see \cite{lassalle2018causal,bartl2021WSSP}. The definition adopted here corresponds to the case of filtrations generated by the processes themselves.} The bicausality constraint corresponds to the fact that
the conditional law of $\pi$ is still a coupling of the conditional laws of $\mu$ and $\nu$, that is $\pi_{x_{1:t},y_{1:t}} \in \cpl\big(\mu_{x_{1:t}}, \nu_{y_{1:t}}\big)$.
This concept
can be expressed in different equivalent ways, see e.g. \cite{backhoff2017causal,acciaio2020causal} in the context of transport, and \cite{bremaud1978changes} in the filtration enlargement framework. 

Restricting the transport problem to only bicausal couplings, we define the adapted Wasserstein distance.
\begin{definition}
For $p\geq 1$, the $p$-th order \textit{adapted Wasserstein distance} $\AW_{p}$ on $\calP_p(\R^{dT})$ is defined by\begin{equation*}\label{eq:adaptedwd}
        \AW_p^{p}(\mu,\nu) = \inf_{\pi \in \bccpl(\mu,\nu)}\int_{\R^{dT}\times\R^{dT}} |x-y|_p^p\, \pi(dx,dy).
    \end{equation*}
\end{definition}

In what follows, we define variational distances via general measurable functions $w : \R^{dT} \to [0,\infty)$, which we call weighting functions, and consider the following weighted versions of the total variation and adapted total variation distances.

\begin{definition}
    \label{def:wTV_wATV}
    Let $\mu,\nu \in \calP(\R^{dT})$ and $w$ be a weighting function. Then the \textit{weighted (adapted) total variation distances} are defined as
    \begin{align}
        \label{eq:def.weighted.TV}
        \TV_w(\mu,\nu) :=& \inf_{\pi \in \cpl(\mu,\nu)} \int_{\{x \neq y\}} (w(x) + w(y)) \, d\pi(x,y), \\
        \label{eq:def.weighted.ATV}
        \ATV_w(\mu,\nu) :=& \inf_{\pi \in \bccpl(\mu,\nu)} \int_{\{x \neq y\}} (w(x) + w(y)) \, d\pi(x,y).
    \end{align}
    In particular,
    \begin{itemize}
        \item when $w\equiv 1$, \eqref{eq:def.weighted.TV}  retrieves the \textit{total variation} 
        $\TV(\mu,\nu) := \int_{\R^{dT}} |\mu - \nu|(dx)$, where $|\mu - \nu| = \mu + \nu - 2(\mu \wedge \nu)$, and \eqref{eq:def.weighted.ATV}
        retrieves the \textit{adapted total variation}
        $\ATV(\cdot,\cdot)$ introduced in \cite{eckstein2024computational};  
        \item when $w(x) = 1+|x|_p^p$, $p \geq 1$, we call the $p$-th root of \eqref{eq:def.weighted.TV} and \eqref{eq:def.weighted.ATV} the \textit{$p$-th order (adapted) total variation distances}, given by
        \begin{align*}
            \TV_p(\mu,\nu) :=& \Big( \inf_{\pi \in \cpl(\mu,\nu)} \int_{\{x \neq y\}} (2 + |x|^p_p + |y|^p_p) \, d\pi(x,y) \Big)^\frac1p, \\
            \ATV_p(\mu,\nu) :=& \Big( \inf_{\pi \in \bccpl(\mu,\nu)} \int_{\{x \neq y\}} (2 + |x|^p_p + |y|^p_p) \, d\pi(x,y) \Big)^\frac1p.
        \end{align*}
    \end{itemize}
\end{definition}
The adapted total variational distance was first introduced in \cite{eckstein2024computational} as a tool to upper-bound the $\AW$-distance in the compact setting. In this work, we extend this concept to define weighted modifications of it, which allow us to establish an upper bound for the $\AW$-distance in the general case; see Theorem~\ref{thm:AWpTVp}.

\subsection{Estimating $\AW$-distance with $\TV$-distance}\label{ssec:bounds ATV}
\label{ssec:main_ATV_TV}
In this subsection, we first establish a precise relationship between the total variation distance and its adapted counterpart. Building on this, we subsequently derive an estimate for the adapted Wasserstein distance in terms of the total variation distance. 
To do so, we first require control over the weighting with respect to one of the marginal measures. This requirement is formalized through the following conditional weighting assumption.
    \begin{assumption}[Conditional weighting] \label{ass:conditional.moments} The distribution $\nu$ and weighting function $w$ satisfy the following condition:  there exist maps $w_t : \R^{dt} \to [0,\infty)$, $t = 1,\ldots,T$, with $w_T = w$, and constants $(c_t)_{t = 2}^T \in \mathbb R_{\geq 0}^{T-1}$ such that, for $\nu$-a.e.\ $x$ and $t = 2,\dots, T$,
\begin{equation}
        \label{eq:ass.cond.moments.1}
        w_{t - 1}(x_{1:t-1}) \leq w_t(x_{1:t})\quad \text{and}\quad    \int w_t(x_{1:t-1},\tilde x_t) \, \nu_{x_{1:t-1}}(d\tilde x_t)\leq (1+c_t) w_{t-1}(x_{1:t-1}).
\end{equation}
    \end{assumption}
Under this assumption, we derive the desired sharp estimate of the adapted total variation distance with respect to the total variation distance; see Sections~\ref{subsect:proofthm25} and \ref{subsec:sharpTVATV} for the proof.
\begin{theorem}\label{thm:weighted_main_inequality}
Let $\mu,\nu \in \calP(\R^{dT})$ and $w$ be a weighting function such that $(\nu,w)$ satisfies Assumption~\ref{ass:conditional.moments} and $\int w(x) \, d(\mu + \nu) < \infty$.
    Then we have
    \begin{equation}
        \label{eq:thm.weighted_main_inequality}
        \ATV_w(\mu,\nu) \le \Big( 1 + 2 \sum_{t = 1}^{T-1} \prod_{s = t + 1}^{T} (1 + c_s) \Big) \TV_w(\mu,\nu),
    \end{equation}
    where $(c_t)_{t = 2}^T$ are the constants in \eqref{eq:ass.cond.moments.1}.
    In particular,
    \begin{itemize}
        \item when $w\equiv 1$, choosing $w_{t-1}(x_{1:t-1}) = 1$ and $c_{t} = 0$, $t=2,\ldots,T$, \eqref{eq:thm.weighted_main_inequality} becomes 
        \begin{equation}
            \label{eq:thm.ATVTV}
            \ATV(\mu,\nu) \le (2T - 1) \TV(\mu,\nu);
        \end{equation}
    \item when $w(x) = 1+|x|_p^p$, $p \geq 1$, choosing
    \begin{equation}
     \label{eq:thm.ATVVp.constraint}
  w_{t-1}(x_{1:t-1}) = 1 + |x_{1:t-1}|_p^p~\text{ and }~  c_t = \nu\text{-}\operatorname{ess}\sup \frac{\mathbb E_\nu[ |X_t|^p | X_1,\ldots,X_{t-1} ]}{1 + \sum_{s=1}^{t-1}|X_{s}|^p}, \quad t=2,\ldots, T,
    \end{equation}
    \eqref{eq:thm.weighted_main_inequality} becomes 
    \begin{equation}
        \label{eq:thm.ATVpTVp}
     \ATV_p^p(\mu,\nu) \le \left( 1+2\sum_{t = 1}^{T-1} \prod_{s = t+1}^T (1 + c_s) \right) \TV_p^p(\mu,\nu).
     \end{equation}
    \end{itemize}
    Moreover, this bound is sharp, i.e., given any $c_t \geq 0$, $t=2,\dots,T$, we can construct $w$ s.t.\ there is a sequence $(\mu_n,\nu_n)_{n \in \mathbb N}\in \calP(\R^{dT})\times\calP(\R^{dT})$ with
    $\lim_{n\to \infty} \frac{\ATV_w(\mu_n,\nu_n)}{\TV_w(\mu_n,\nu_n)} = \Big( 1 + 2 \sum_{t = 1}^{T-1} \prod_{s = t + 1}^{T} (1 + c_s) \Big)$; see Theorem~\ref{ex:boundissharp}.
\end{theorem}
Notoriously, the set $\bccpl(\mu,\nu)$ of bicausal couplings can be considerably smaller than the set $\cpl(\mu,\nu)$ of all couplings. This makes the above result quite surprising, as we obtain that the $\ATV$ distance is bounded by the $\TV$ distance up to a constant that depends linearly on the number of time steps. However, between $\ATV_{p}$ and $\TV_{p}$, in general there exists no such estimate with constant scaling linearly in time. More precisely, for all $p \geq 1$, there exists no constant $C = C(T,p)$ such that the inequality $\ATV_{p}(\mu,\nu) \leq C \TV_{p}(\mu,\nu)$ holds for all $\mu, \nu \in \calP(\R^{dT})$. 
Therefore, $C$ has to depend on the sequence $(c_t)_{t=2}^{T}$ defined in \eqref{eq:thm.ATVVp.constraint}; see Example~\ref{ex:ct} for a counterexample.

\begin{remark}
    Similar to Wasserstein distances, one can define $\TV_w$ and $\ATV_w$ on general Polish spaces. Then the estimate in \eqref{eq:thm.weighted_main_inequality} still holds, with the same proof given in Section~\ref{sec:ATV_TV}, just adapting the notation to that of Polish spaces. 
\end{remark}

Building on the results above, we now provide an estimate for the adapted Wasserstein distance in terms of the total variation distance; see Subsection~\ref{subsec:ATV_TV} for the proof. 

\begin{corollary}
\label{thm:AWpTVp}
    Let $p\geq 1$, $\mu,\nu\in\calP_p(\R^{dT})$, and $(c_t)_{t = 2}^T$ as in \eqref{eq:thm.ATVVp.constraint}. Then
    \begin{equation*}
        \AW_p^p(\mu,\nu) \leq 2^p\ATV_p^p(\mu,\nu) \leq 2^p\Big( 1 + 2 \sum_{t = 1}^{T-1} \prod_{s = t + 1}^{T} (1 + c_s) \Big)\TV_p^p(\mu,\nu),
    \end{equation*}
    If, moreover, $\mu,\nu \in \calP_p(\calK)$, with $\calK\subseteq \mathcal B(\R^{dT})$ such that $\mathrm{diam}(\calK) := \max_{x,y \in \calK}|x-y|<\infty$, then
        \begin{equation*}
        \AW_p^p(\mu,\nu) \leq \mathrm{diam}(\calK)^p\ATV(\mu,\nu) \leq (2T - 1)\mathrm{diam}(\calK)^p\TV(\mu,\nu).
    \end{equation*}  
\end{corollary}

\subsection{Estimating $\AW$-distance with $\W$-distance under Sobolev regularity}
\label{ssec:main_AW_W}
For probability measures on $\R^{dT}$ with densities having finite Sobolev norm, we provide estimates of the adapted Wasserstein distance in terms of the classical Wasserstein distance. The result is stated below, and the proof deferred to Section~\ref{sec:AW_W}.

\begin{definition}
\label{def:sob}
    For $k \in \N_0$ and $r\in [1, \infty]$, the Sobolev space $W^{k,r}(\R^{dT})$ is defined as 
    \begin{equation*}
        W^{k,r}(\R^{dT}) = \left\{u \in L^r(\R^{dT}) : 
         D^{{\alpha}}u\in L^r(\R^{dT})\; \forall\alpha \in \N^d_0 \text{ with } |{\alpha}| \leq k\right\},
    \end{equation*}
    and equipped with the Sobolev norm $\Vert u \Vert_{W^{k,r}} = \Vert u \Vert_{k,r} =  \sum_{|{\alpha}|\leq k} \Vert D^{{\alpha}}u \Vert_{r}$,\footnote{Here, we adopt the notion of Sobolev norm used in \cite{chae2020wasserstein}.} so, in particular, $\Vert \cdot \Vert_{0, r} = \Vert \cdot \Vert_r$.
\end{definition}
\begin{definition}
\label{def:kthorderkernel}
Let \( k \in \mathbb{N} \). A function \( K \colon \mathbb{R}^{dT} \to  \R \) is called a \emph{\(k\)-th order kernel} if it satisfies:
\begin{enumerate}[(i)]
    \item $\int_{\mathbb{R}^{dT}} K(x) \, dx = 1,$
    \item $\int_{\mathbb{R}^{dT}} K(x)\, x^{\alpha} \, dx = 0$ for all multi-indices $\alpha \in \mathbb{N}_0^{dT}$ with $1 \leq |\alpha| < k.$
\end{enumerate}
\end{definition}

\begin{theorem}
\label{thm:AWtoW}
Let \( p, q \geq 1 \), \( k \in \mathbb{N} \), and let \( K \) be a kernel of order \( k \). Suppose \( \mu, \nu \in \mathcal{P}(\mathbb{R}^{dT}) \) have densities \( f, g \in W^{k,1}(\mathbb{R}^{dT}) \), respectively. Then
\begin{equation}
\label{eq:thm:AWtoW.1}
    \begin{split}
        \AW_{p}^p(\mu,\nu) \leq C_0 C_1\W_q(\mu,\nu) + 2C_0\big(C_{k,K}^{\frac{1}{k}}C_2\big)^{\frac{k}{k+1}}\left(\|f_p \|_{k,1} + \|g_p \|_{k,1}\right)^{\frac{1}{k+1}}\W_q^{\frac{k}{k+1}}(\mu,\nu),
    \end{split}
\end{equation}
where $f_p(x) = (1+|x|_p^p)f(x)$, $g_p(x) = (1+|x|_p^p)g(x)$, and
\begin{align*}
    &C_0 \coloneqq 2^p\Big( 1 + 2 \sum_{t = 1}^{T-1} \prod_{s = t + 1}^{T} (1 + c_s) \Big),\qquad\quad
    C_1 \coloneqq \| K \|_1 p\left(M_{\frac{q(p-1)}{q-1}}^{\frac{q-1}{q}}(\mu) + M_{\frac{q(p-1)}{q-1}}^{\frac{q-1}{q}}(\nu)\right),\\
    &C_2 \coloneqq \mathrm{Lip}(K) \left(1+ p M_{\frac{qp}{q-1}}^{\frac{q-1}{q}}(\mu) + (p+1)M_{\frac{qp}{q-1}}^{\frac{q-1}{q}}(\nu)\right), ~ C_{k,K} \coloneqq \sup_{|\alpha| \leq k}\frac{1}{\alpha!}\int |K(z)z^\alpha|dz,
\end{align*}
with $(c_t)_{t = 2}^T \in R^{T-1}$ as in \eqref{eq:thm.ATVVp.constraint}.
In particular,
\begin{enumerate}[(i)]
    \item when $\mu,\nu \in \calP(\calK)$ with $\calK \subseteq \R^{dT}$ compact, we have
    \begin{equation}
    \label{eq:thm:AWtoW.2}
    \AW_{p}^p(\mu,\nu) \leq 2(2T - 1)\mathrm{diam}(\calK)^p\big(C_{k,K}^{\frac{1}{k}}\mathrm{Lip}(K)\big)^{\frac{k}{k+1}} \left(\|f \|_{k,1} + \|g \|_{k,1}\right)^{\frac{1}{k+1}} \W_1^{\frac{k}{k+1}}(\mu,\nu);
    \end{equation}
    \item when $\mu,\nu \in \calP(\calK)$ with $\calK \subseteq \R^{dT}$ compact, and $f,g$ are polynomial densities of order less than $k$ on $\calK$,
    \begin{equation}
    \label{eq:thm:AWtoW.3}
        \AW_p^p(\mu,\nu) \leq (2T - 1)\mathrm{diam}(\calK)^p\mathrm{Lip}(K) \W_1(\mu,\nu).
    \end{equation}
\end{enumerate}
Moreover, the order in \eqref{eq:thm:AWtoW.1} is sharp for $\AW_1$; see Example~\ref{ex:orderissharp}.
\end{theorem}

\subsection{Fast rate estimator for $\AW_1$-distance}
In this subsection, we construct an estimator of the $\AW_1$-distance which is based on kernel density estimation, and whose rate of convergence improves as the smoothness \( k \) of the underlying measure \( \mu \) increases. 
The result is stated below, and the proof deferred to Section~\ref{sec:fast}. We briefly introduce some notation from the theory of wavelets, and refer the reader to 
Appendix~\ref{app:wavelets} for details and references. 
To define a basis over the unit cube $[0,1]^{dT}$, we focus on the boundary-corrected 
$N$-th Daubechies wavelet system, for an integer $N \ge 2$, as introduced by \cite{cohen1993wavelets}. 
In short, given an integer $j_0 \ge \log_2 N$, their construction leads to respective families 
of scaling and wavelet functions 
\[
\Phi
= \{ \zeta_{j_0k} : 0 \le k \le 2^{j_0} - 1 \}, 
\quad
\Psi_j 
= \{ \xi_{jk\ell} : 0 \le k \le 2^{j_0} - 1,\, \ell \in \{0,1\}^d \setminus \{0\} \}, 
\quad j \ge j_0,
\]
such that $\Psi = \Phi \cup \bigcup_{j = j_0}^{\infty} \Psi_j$ forms an orthonormal basis of $L^2([0,1]^{dT})$, 
with the property that $\Phi$ spans all polynomials of degree at most $N - 1$ over $[0,1]^{dT}$. 

\begin{definition}[Besov space]
\label{def:besov}
Let $s > 0$, and $f \in L^\infty([0,1]^{dT})$ admit the wavelet expansion
\[
f = \sum_{\zeta \in \Phi} \beta_\zeta \zeta + \sum_{j=j_0}^{\infty}\sum_{\xi \in \Psi_j}\beta_\xi \xi\quad \text{over $[0,1]^{dT}$,}  
\]
with convergence in $L^\infty([0,1]^{dT})$, where $\beta_\xi = \int \xi(x) f(x)dx$  for all $\xi \in \Psi$. The Besov norm of $f$ is defined by 
\[
\Vert f \Vert_{\mathcal{B}_{\infty,\infty}^{s}([0,1]^{dT})} = \Vert (\beta_\zeta)_{\zeta \in \Phi} \Vert_{\infty} + \left\Vert \left(2^{j(s+\frac{dT}{2})}\Vert (\beta_\xi)_{\xi \in \Psi_j}\Vert_{\infty}\right)_{j \geq j_0}\right\Vert_{\infty},
\]
and we set
\[
\mathcal{B}_{\infty, \infty}^s([0,1]^{dT}) = \left\{f \in L^\infty([0,1]^{dT})\colon \Vert f \Vert_{\mathcal{B}_{\infty,\infty}^s([0,1]^{dT})} < \infty \right\}.
\]
\end{definition}
For $\mu \in \calP(\R^{dT})$ and $n\in\N$, let $(X^{(i)})_{i=1}^{n}$ be i.i.d.\ samples of $\mu$. Define the empirical measure $\mu^n = \frac{1}{n}\sum_{i=1}^{n}\delta_{X_i}$ and the wavelet density estimator
\[
\tilde f_n = \sum_{\zeta \in \Phi}\hat \beta_\zeta^n \zeta + \sum_{j=j_0}^{J_n} \sum_{\xi \in \Psi_{j}} \hat\beta^n_\xi \xi,
\]
where $J_n \geq j_0$ is a deterministic threshold, and  $\hat \beta_\zeta^n = \int \zeta d\mu^n$,  $\hat\beta^n_\xi = \int \xi d\mu^n$ for all $\zeta \in \Phi$, $\xi \in \Psi_j$, $j_0 \leq j \leq J_n$. Let \(\varphi\) denote the standard Gaussian density function. We define the random measure
\begin{equation}\label{eq:munhat}
\hat \mu_n(dx) = \hat f_n(x)dx,
\end{equation}
where
\[
\hat f_n(x) = 
\begin{cases}
    \tilde f_n(x), &  \text{if $\tilde f_n(x)$ is non-negative,}\\
    \varphi(x - X^{(1)}),  & \text{otherwise}.
\end{cases}
\]
We adopt the notation  \( a \lesssim b \) to indicate that there exists a constant \( C \) independent of the number of samples $n$, such that \( a \leq Cb \), and write \( a \simeq b \) to indicate that \( a \lesssim b \) and \( b \lesssim a \).

\begin{theorem}
\label{thm:fast_rate_compact}
Let $dT\geq 3$, $s > k \in \N$, $\gamma > 0$, $\mu \in \calP([0,1]^{dT})$ with density $f \in \mathcal{B}^s_{\infty, \infty}([0,1]^{dT})$. Assume further that $f \geq 1/\gamma$ over $[0,1]^{dT}$, and let $2^{J_n} \simeq n^{1/(2s+dT)}$. Then
\begin{equation}
\label{eq:fast_rate}
    \E\big[\AW_1(\mu,\hat \mu_n)\big] \lesssim n^{-\frac{k}{2k + dT}\frac{k}{k+1}}.
    \end{equation}
\end{theorem}
When the dimension $dT$ is fixed, increasing the smoothness parameter $k$ leads to an improved convergence rate---transitioning from the slow, dimension-dependent rate to the Monte-Carlo rate (\( n^{-1/2} \)).
\begin{remark}
    Let \(\varphi\) denote the standard Gaussian density function, and  define the kernel density estimator by $
    \tilde{f}_{n,h}(x) = \frac{1}{n} \sum_{i=1}^{n} \frac{1}{h} \varphi\left( \frac{x - X^{(i)}}{h} \right),
    $ which induces the Gaussian-smoothed empirical measure \(\tilde{\mu}_{n,h}\) with density \(\tilde{f}_{n,h}\). The convergence rate of $\tilde{\mu}_{n,h}$ is established in \cite{hou2024convergence} when $\mu$ has Lipschitz conditional kernels.
    For $L>0$ and $\mu \in \calP_1(\R^{dT})$, $\mu\in\mathcal{P}_1(\R^{dT})$ is said to have $L$-Lipschitz conditional kernels if it admits a disintegration such that, for all $t=1,\dots,T-1$, $x_{1:t} \mapsto \mu_{x_{1:t}}$ is $L$-Lipschitz (where $\calP(\R^{d})$ is equipped with $\W_1$). Then, if $h \simeq n^{-\frac{1}{dT+2}}$, we have
    \begin{equation}
        \label{eq:slow_rate}
        \E\big[\AW_{1}(\mu , \tilde{\mu}_{n,h}) \big] \lesssim n^{-\frac{1}{dT+2}}.
    \end{equation}
    Comparing \eqref{eq:fast_rate} and \eqref{eq:slow_rate} for $dT\ge 3$, \eqref{eq:slow_rate} is faster when $k=1$, whereas \eqref{eq:fast_rate} is faster when $k\ge 3$. For $k = 2$, \eqref{eq:slow_rate} is faster at $dT = 3$, the two rates coincide at $dT = 4$, and \eqref{eq:fast_rate} becomes faster for $dT > 4$. To summarize, \eqref{eq:fast_rate} and \eqref{eq:slow_rate} rely on different assumptions, and neither rate dominates the other globally for all $k\geq 1$ and $dT\geq 3$.
\end{remark}

\section{Linking $\ATV$-distance with $\TV$-distance}
\label{sec:ATV_TV}
The goal of this section is to prove Theorem~\ref{thm:weighted_main_inequality} and Theorem~\ref{thm:AWpTVp}.
To this end, we now introduce a probabilistic setting which allows us to work with density processes associated with the marginals $\mu$ and $\nu$.
More precisely, we consider the canonical setting where $\Omega=\R^{dT}$ is the path space, $X$ the canonical process, $\mathbb F=(\mathcal F_t)_{t=1}^T$ the canonical filtration generated by $X$, and $\mathbb P$ a probability measure on $(\Omega, \mathcal F)$, with $\mathcal F=\mathcal F_T$, that dominates both $\mu$ and $\nu$. We denote the random variables associated with the respective densities by 
\[ Z^1\coloneqq\dfrac{d\mu}{d\mathbb P} \quad\text{and}\quad Z^2\coloneqq\dfrac{d\nu}{d\mathbb P}, \]
and define the density processes $(Z^i_t)_{t=1}^{T}$, $i=1,2$, so that $Z^i_t=\mathbb E[Z^i|\mathcal F_t]$, with the convention that $Z^i_0=1$.
Moreover, we define the random variables associated with the conditional densities 
\[ D^i_t=\frac{Z^i_t}{Z^i_{t-1}}, \quad t =1,\ldots,T,\quad i=1,2,\]
with the convention that $D_t^i = 1$ if $Z_{t - 1}^i = 0$,
which are $\mathbb P$-almost surely well-defined. 
Intuitively speaking, $Z_t^1$ is the density associated with the marginal distribution $\mu_{1:t}$ and $D_t^1$ is density associated with the conditional distribution $\mu_{X_{1:t-1}}$, $t = 1,\dots,T$, and analogously for $Z_t^2$ and $D_t^2$.
    
\subsection{Representation of $\TV$-distance and $\ATV$-distance}\label{subsec.ATV}

Let $w : \R^{dT} \to [0,\infty)$ be a general weighting function, used to define $\TV_w$ and $\ATV_w$ distances in Definition~\ref{def:wTV_wATV}. We can explicitly solve the optimal transport problem of $\TV_w$ and $\ATV_w$, i.e. \eqref{eq:def.weighted.TV} and \eqref{eq:def.weighted.ATV}. This enables us to rewrite these distances via the density processes associated with $\mu$ and $\nu$.
    \begin{lemma}
        \label{lem:TV.charlight}
        Let $\mu,\nu \in \mathcal P(\R^{dT})$. Then
        \begin{equation}
            \label{eq:lem.TV.charlight.0}
            \TV_w(\mu,\nu) = \int w\ d\mu+\int w\ d\nu-2\int w\ d(\mu\wedge\nu) = \E\Big[ w(X) \big(Z^1 + Z^2 - 2\min(Z^1, Z^2)\big)\Big],
        \end{equation}
        and $\pi^*$ is an optimizer of \eqref{eq:def.weighted.TV} if and only if $\pi^* = (\mathbf{id},\mathbf{id})_\# (\mu\wedge \nu)$ on $\{x=y\}$.
    \end{lemma}
    \begin{proof}
        Notice that
        \begin{equation}
        \label{eq:lem.TV.charlight.1}
        \begin{split}
            \TV_w(\mu,\nu) 
            &= \inf_{\pi \in \cpl(\mu,\nu)} \int_{\{x \neq y\}} (w(x) + w(y)) \, d\pi(x,y)\\
            &=  \int w\ d\mu+\int w\ d\nu - \sup_{\pi \in \cpl(\mu,\nu)}\int 2w(x)\mathbbm{1}_{\{x = y\}} \, d\pi(x,y)\\
            &=  \int w\ d\mu+\int w\ d\nu-2\int w\ d(\mu\wedge\nu) = \E\Big[ w(X) \big(Z^1 + Z^2 - 2\min(Z^1, Z^2)\big)\Big],
        \end{split}
        \end{equation}
        where the supremum in \eqref{eq:lem.TV.charlight.1} is attained by any $\pi^\star$ that satisfies $\pi^\star = (\mathbf{id},\mathbf{id})_\# (\mu\wedge \nu)$ on $\{x=y\}$.
    \end{proof}
    Next, we extend Lemma~\ref{lem:TV.charlight} to its adapted counterpart, by employing the dynamic programming principle (DPP) of bicausal optimal transport; see Proposition~5.2. in \cite{backhoff2017causal}.
    \begin{lemma}\label{lem:ATV.charlight}
        Let $\mu,\nu \in \mathcal P(\R^{dT})$. Then
        \begin{equation}
        \label{eq:lem.ATV.charlight.0}
        \begin{split}
            \ATV_w(\mu,\nu) &=
            \int w \, d\mu + \int w \, d\nu - 2\int w\, d\Big( 
            \mu_1 \wedge \nu_1 \prod_{t = 1}^{T-1} \mu_{x_{1:t}} \wedge \nu_{y_{1:t}}\Big)\\
            &= \mathbb E\Big[ w(X) \Big( Z^1 + Z^2 - 2 \prod_{t = 1}^T \min(D^1_t, D^2_t) \Big) \Big],
        \end{split}
        \end{equation}
        and $\pi^*= \prod_{t=0}^{T-1}\pi^*_{x_{1:t},y_{1:t}}$ is an optimizer of \eqref{eq:def.weighted.ATV} if and only if $\pi^*_{x_{1:t},x_{1:t}} = (\mathbf{id},\mathbf{id})_\# (\mu_{x_{1:t}}\wedge \nu_{x_{1:t}})$ on $\{x_{t+1}=y_{t+1}\}$ for all $x_{1:t} \in \R^{dt}$, $t=0,\ldots, T-1$.
    \end{lemma}
    \begin{proof}
    Similar to the proof of Lemma~\ref{lem:TV.charlight}, we notice that 
    \begin{equation}
    \label{eq:lem.ATV.charlight.0.5}
    \begin{split}
        \ATV_w(\mu,\nu) 
        &=  \int w(x)d\mu(x)+\int w(y)d\nu(y) - \sup_{\pi \in \bccpl(\mu,\nu)}\int 2w(x)\mathbbm{1}_{\{x = y\}} \, d\pi(x,y).
    \end{split}
    \end{equation}
    Therefore, it is sufficient to prove that
    \begin{equation}
    \label{eq:lem.ATV.charlight.1}
    \begin{split}
        \sup_{\pi \in \bccpl(\mu,\nu)}\int w(x)\mathbbm{1}_{\{x = y\}} \, d\pi(x,y) = \int w(x)\, d\Big( 
            \mu_1 \wedge \nu_1 \prod_{t = 1}^{T-1} \mu_{x_{1:t}} \wedge \nu_{y_{1:t}}\Big)(x),
    \end{split}
    \end{equation}
    as this readily implies \eqref{eq:lem.ATV.charlight.0}. In what follows, we prove \eqref{eq:lem.ATV.charlight.1} by induction: When $T=1$, $\ATV = \TV$, so \eqref{eq:lem.ATV.charlight.1} holds directly by Lemma~\ref{lem:TV.charlight}. Now, assuming that \eqref{eq:lem.ATV.charlight.1} holds for $T-1$, then
    \begin{equation*}
        \begin{split}
            &\quad~ \sup_{\pi \in \bccpl(\mu,\nu)}\int w(x)\mathbbm{1}_{\{x = y\}} \, d\pi(x,y)\\
            &\stackrel{\text{(by DPP)}}{=} \sup_{\pi_1 \in \cpl(\mu_1,\nu_1)} \int \mathbbm{1}_{\{x_1 = y_1\}}\Big[ \sup_{\pi_{x_1,y_1} \in \cpl(\mu_{x_1},\nu_{y_1})} \int  w(x)\mathbbm{1}_{\{x_{2:T} = y_{2:T}\}}d\pi_{x_1,y_1}(x_{2:T},y_{2:T})\Big]d\pi_1(x_1,y_1)\\
            &\stackrel{\text{(by assumption)}}{=} \sup_{\pi_1 \in \cpl(\mu_1,\nu_1)} \int \mathbbm{1}_{\{x_1 = y_1\}}\Big[ \int w(x_1,x_{2:T})\, d\Big( 
            \mu_{x_1} \wedge \nu_{y_1} \prod_{t = 2}^{T-1} \mu_{x_{1:t}} \wedge \nu_{y_{1:t}}\Big)(x_{2:T})\Big]d\pi_1(x_1,y_1) \\
            &= \int w(x)\, d\Big( 
            \mu_1 \wedge \nu_1 \prod_{t = 1}^{T-1} \mu_{x_{1:t}} \wedge \nu_{y_{1:t}}\Big)(x) \quad \text{(by Lemma~\ref{lem:TV.charlight})},
        \end{split}
    \end{equation*}
    where the supremum in the last equality is attained by any $\pi_1^*$ that satisfies $\pi_1^* = (\mathbf{id},\mathbf{id})_\# (\mu_1\wedge \nu_1)$ on $\{x_1=y_1\}$. By rewriting the equation \eqref{eq:lem.ATV.charlight.0.5} with density processes, we get
    \begin{multline*}
        \int w \, d\mu + \int w \, d\nu - \int 2w\, d\Big(\mu_1 \wedge \nu_1 \prod_{t = 1}^{T-1} \mu_{x_{1:t}} \wedge \nu_{y_{1:t}}\Big)\\
        = \mathbb E\Big[ w(X) \Big( Z^1 + Z^2 - 2 \prod_{t = 1}^T \min(D^1_t, D^2_t) \Big) \Big].
    \end{multline*}
    Therefore, we establish \eqref{eq:lem.ATV.charlight.0} and the characterization of optimizers follows analogously.
    \end{proof}
\begin{remark}
    The proof of Lemma~\ref{lem:TV.charlight} (resp. Lemma~\ref{lem:ATV.charlight}) generalizes to the following setting:
    Let $w_1\colon \R^{dT} \to [0,\infty), w_2 \colon \R^{dT} \to [0,\infty)$ be non-negative weighting functions. Then we have
    \begin{align*}
        \inf_{\pi \in \cpl(\mu,\nu)} \int_{\{x \neq y\}} w_1(x) + w_2(y) \, d\pi(x,y) = \int w_1 \, d\mu + \int w_2 \, d\nu - \int w_1 + w_2 \, d(\mu \wedge \nu),
    \end{align*}
    and
    \begin{multline*}
        \inf_{\pi \in \bccpl(\mu,\nu)} \int_{\{x \neq y\}} w_1(x) + w_2(y) \, d\pi(x,y)\\
        = \int w_1 \, d\mu + \int w_2 \, d\nu - \int w_1 + w_2 \, d\Big(\mu \wedge \nu \prod_{t = 1}^{T-1} \mu_{x_{1:t}} \wedge \nu_{x_{1:t}} \Big).
    \end{multline*}
\end{remark}

\subsection{Proof of the estimate \eqref{eq:thm.weighted_main_inequality} in Theorem~\ref{thm:weighted_main_inequality} and of Theorem~\ref{thm:AWpTVp}}\label{subsect:proofthm25}

\label{subsec:ATV_TV}

Now, in order to prove Theorem~\ref{thm:weighted_main_inequality}, our goal is to find $\lambda\in \R$ such that 
\[
\min_{\mu,\nu}\lambda\TV_w(\mu,\nu) - \ATV_w(\mu,\nu) \geq 0.
\]
Notice that, since $Z^1 + Z^2 = |Z^1 - Z^2| + 2\min(Z^1, Z^2)$, we can rewrite $\lambda\TV_w(\mu,\nu) - \ATV_w(\mu,\nu)$ using the representations established in Section~\ref{subsec.ATV}:
\begin{equation*} 
    \begin{split}
        &\quad~\lambda\TV_w(\mu,\nu) - \ATV_w(\mu,\nu) \\
        &= \lambda\, \mathbb E\Big[ w(X) \Big( Z^1 + Z^2 - 2\min(Z^1, Z^2) \Big) \Big] - \mathbb E\Big[ w(X) \Big( Z^1 + Z^2 - 2 \prod_{t = 1}^T \min(D^1_t, D^2_t) \Big) \Big]\\
        &= \mathbb E\Big[ w(X) \Big( (\lambda - 1) | Z^1 - Z^2 | + 2 \prod_{t = 1}^T \min(D^1_t, D^2_t) - 2 \min(Z^1, Z^2)  \Big) \Big].
    \end{split}
\end{equation*}
For now we assume $\mu \ll \nu$ so that all notation below are well defined. Recalling that $Z^i = \prod_{t=1}^{T}D_t^i$, $i=1,2$, we can further rewrite 
\begin{equation*}
\begin{split}
    &\quad~\lambda\TV_w(\mu,\nu) - \ATV_w(\mu,\nu)\\
    &= \mathbb E\Big[ w(X) \Big( (\lambda - 1) \Big| \frac{Z^1}{Z^2} - 1 \Big| + 2 \prod_{t = 1}^T\frac{ \min(D^1_t, D^2_t)}{D_t^2} - 2 \min\Big(\frac{Z^1}{Z^2}, 1\Big)  \Big) Z^2\Big]\\
    &= \mathbb E\Big[ w(X) \Big( (\lambda - 1) \Big| \prod_{t=1}^{T}\frac{D_t^1}{D_t^2} - 1 \Big| + 2 \prod_{t = 1}^T \min\Big(\frac{D^1_t}{D^2_t}, 1\Big) - 2 \min\Big(\prod_{t=1}^{T}\frac{D_t^1}{D_t^2}, 1\Big)  \Big) Z^2\Big].
\end{split}
\end{equation*}
When $T=1$, it is obvious that $\lambda =1$ is sufficient to guarantee non-negativity because in this case $\prod_{t = 1}^T \min\Big(\frac{D^1_t}{D^2_t}, 1\Big) = \min\Big(\prod_{t=1}^{T}\frac{D_t^1}{D_t^2}, 1\Big)$. 
When $T>2$, we deploy the dynamic programming structure
in order to find the optimal $\lambda$ by induction. For notational simplicity, we set, for $t=1,\dots, T$,
\[
Y_t := \frac{D^1_t}{D^2_t},~
A_{T-1} := \prod_{t = 1}^{T-1} \min( Y_t, 1 ), ~
B_{T-1} := \frac{Z^1_{T-1}}{Z^2_{T-1}} = \prod_{t = 1}^{T-1} Y_t,~ \text{so that } B_{T-1}Y_{T}=\prod_{t=1}^{T}\frac{D_t^1}{D_t^2}.
\]
Then we can write 
\begin{equation}
\label{eq:lambdaTVwATVw}
\begin{split}
&\quad~\lambda\TV_w(\mu,\nu) - \ATV_w(\mu,\nu) \\
    &= \mathbb E\Big[ w(X) \Big( (\lambda - 1) |1-B_{T-1}Y_{T}| + 2 A_{T-1} \min(1,Y_T) - 2 \min(1,B_{T-1}Y_{T})  \Big) D_T^2 Z^2_{T-1}\Big].
\end{split}
\end{equation}
Next,  for $a,b,y \in \R$, $x_{1:T} \in \R^{dT}$, we define the auxiliary functions
\[
J(x_T,y;x_{1:T-1},a,b) \coloneqq w(x_{1:T}) \Big( (\lambda - 1) |1-by| + 2 a \min(1,y) - 2 \min(1,by)  \Big).
\]
By the tower property, we then have 
\begin{equation}
\label{eq:lambdaTVwATVw.1}
\begin{split}
    \lambda\TV_w(\mu,\nu) - \ATV_w(\mu,\nu)
    &=  \mathbb E\Big[ J(X_T,Y_T;X_{1:T-1},A_{T-1},B_{T-1})D_T^2 Z^2_{T-1}\Big]\\
    &= \mathbb E\Big[ \mathbb E\Big[J(X_T,Y_T;X_{1:T-1},A_{T-1},B_{T-1})D_T^2\Big|\calF_{T-1}\Big]Z^2_{T-1}\Big]\\
    &= \mathbb E_{\nu_{1:T-1}}\Big[ \underbrace{\mathbb E_{\nu_{x_{1:T-1}}}\Big[J(X_T,Y_T;X_{1:T-1},A_{T-1},B_{T-1})\Big]}_{(\star)} \Big].
\end{split}
\end{equation}
In the first step of the induction, we focus on minimizing the inner expectation $(\star)$. 
Note that $A_{T-1} \in [0,1]$, $B_{T-1}\geq 0$, $A_{T-1} \leq B_{T-1}$, $\nu_{x_{1:T-1}}$ is a distribution in $\calP(\R^{d})$, and $Y_T$ is a non-negative random variable satisfying $\E_{\nu_{x_{1:T-1}}}[Y_T] = 1$. 
Thus, to find the worst-case lower bound in $(\star)$,
we minimize $\mathbb E_\eta[J(X_T,Y_T;x_{1:T-1},a,b)]$ over all $\eta \in \calP(\R^d)$, $b \ge 0$, $0 \le a \le \min(1,b)$, and $Y_T \ge 0$ with  $\E_\eta[Y_T]  = 1$. 
Therefore, we consider the following minimization problem:
\begin{equation}
    \label{eq:toshow.bound.wmoments}
    \min_{ \eta \in \calP(\R^d), Y_T\geq 0, \E_\eta[Y_T] = 1} \E_{\eta}\Big[w(x_{1:T-1},X_T) \Big( (\lambda - 1) |1-bY_T| + 2 a \min(1,Y_T) - 2 \min(1,bY_T)  \Big)\Big].
\end{equation}
Before diving into solving this problem, we need one further structural assumption on the weighting term $w(x_{1:T-1},X_T)$.
Since we want to find the parameter $\lambda$ such that \eqref{eq:lambdaTVwATVw.1} is non-negative by induction, after solving the first minimization problem \eqref{eq:toshow.bound.wmoments}, we want to get a new weighting term ``$w_{T-1}(x_{1:T-1})$" so that the induction can proceed. 

If $w(x_{1:T}) = 1+\sum_{t=1}^{T}|x_t|^p$, it is clear that we can define $w_{T-1}(x_{1:T-1}) = 1+\sum_{t=1}^{T-1}|x_t|^p$. However, in the generality in which Theorem~\ref{thm:weighted_main_inequality} is stated, $w$ does not have to be of this form.
For this reasons we assume the existence of a sequence of weighting terms $(w_t)_{t=1}^{T}$ in Assumption~\ref{ass:conditional.moments}. 
Given $x_{1:T-1}$, we let $l \coloneqq w_{T-1}(x_{1:T-1})$ and $\Delta W \coloneqq w_T(x_{1:T-1}, X_T) - w_{T-1}(x_{1:T-1})$. 
For $t = T$ in \eqref{eq:ass.cond.moments.1}, we find that $(\nu_{x_{1:T-1}},\Delta W)$ satisfies $\Delta W \geq 0$ and $\E_{\nu_{x_{1:T-1}}}[\Delta W] \leq c_{T}l$. Therefore, to obtain the worst-case bound, similar to the derivation of \eqref{eq:toshow.bound.wmoments} we minimize over all admissible tuples $(\eta,\Delta W, Y_T)$, that is,
\begin{equation}
    \label{eq:toshow.bound.wmoments.general}
    \min_{ \substack{ \eta \in \calP(\R^d), Y_T,\Delta W \geq 0,\\ \E_\eta[Y_T] = 1, \, \E_\eta[\Delta W] \le c_T l}} \E_\eta[ H(l,\Delta W,\lambda,1,a,b,Y_T)],
\end{equation}
where the integrand $H$ is of slightly more general form than in \eqref{eq:toshow.bound.wmoments}.
The function $H\colon \R_+^7 \to \R$ is given by
\begin{equation}
\label{eq:toshow.bound.general_form}
    H(l,u,\lambda,\kappa,a,b,y) = (l + u) \left((\lambda - \kappa)|1-by| + 2\kappa \big( a \min(1,y) - \min(1,by) \big)\right).
\end{equation}
This will eventually allow us to conduct an induction over the time steps, which is essential for the proof of the main result.

\begin{lemma}\label{lem:ATTV}
Let $a,b,c,l,\lambda,\kappa \in \mathbb R^+$ with $0 \le a \le \min(1,b)$ and $\lambda \ge \kappa$.
Then 
\begin{equation}
\label{eq:lem.ATTV}
\inf_{\substack{Y, \,\Delta W \ge 0, \\ \mathbb E[Y] = 1,\, \mathbb E[\Delta W] \le c l} } 
\mathbb E\left[ H(l,\Delta W,\lambda,\kappa,a,b,Y) \right]
=l \cdot 
\begin{cases}
|1 - b|(\lambda - \kappa) + 2\kappa (b + c) ( a - 1), & b \le 1, \\
|1 - b|(\lambda - \kappa) + 2\kappa (1 + c) ( a / b - 1), & b \ge 1.
\end{cases}
\end{equation}
\end{lemma}
    \begin{proof}
        Consider the function $f$ given by
        \[
            f(x,y) := (l + x) \big( (\lambda - \kappa) |1 - by| + 2 \kappa \big( a \min(1,y) - \min(1,by) \big)\big),
        \]
        for $x,y \ge 0$,
        and note that the left-hand side in \eqref{eq:lem.ATTV} is precisely $f^{\ast\ast}(cl,1)$, that is the value of the convex hull of $f$ at $(cl,1)$.
    \begin{figure}[H]     
    \begin{center}
    \begin{minipage}{0.5\textwidth}
    \small
    \begin{tikzpicture}
    \begin{axis}[
        width=8cm, height=6cm,
        grid=both,
        grid style={line width=.1pt, draw=gray!10},
        major grid style={line width=.2pt,draw=gray!50},
        axis lines=left,
        xlabel={$y$},
        ylabel={},
        xmin=0, xmax=3,
        ymin=-9, ymax=13,
        legend style={draw=none, legend pos=north west},
        title={Case $b \leq 1$}
    ]
    \addplot[
        blue,
        thick
    ]
    coordinates {
        (0,4) (1,2) (2,-2) (4,2)
    };
    \addlegendentry{$f(0, y)$}

    \addplot[
        orange,
        thick
    ]
    coordinates {
        (0,4) (2,-2) (4,-8)
    };
    \addlegendentry{$g(0, y)$}
    \end{axis}
\end{tikzpicture}
\end{minipage}%
\begin{minipage}{0.5\textwidth}
    \begin{tikzpicture}
    \small
    \begin{axis}[
        width=8cm, height=6cm,
        grid=both,
        grid style={line width=.1pt, draw=gray!10},
        major grid style={line width=.2pt,draw=gray!50},
        axis lines=left,
        xlabel={$y$},
        ylabel={},
        xmin=0, xmax=3,
        ymin=-3, ymax=13.5,
        legend style={draw=none, legend pos=north west},
        title={Case $b \geq 1$}
    ]
    \addplot[
        blue,
        thick
    ]
    coordinates {
        (0,2.5) (0.5,-1) (1,1) (4, 10)
    };
    \addlegendentry{$f(0, y)$}

    \addplot[
        orange,
        thick
    ]
    coordinates {
        (0,-2.5) (0.5,-1) (4, 9.5)
    };
    \addlegendentry{$g(0, y)$}
    \end{axis}
\end{tikzpicture}
\end{minipage}
\end{center}
\caption{Visualization of $f(0,\cdot)$ and $g(0,\cdot)$.}
\label{fig:fg}
\end{figure}
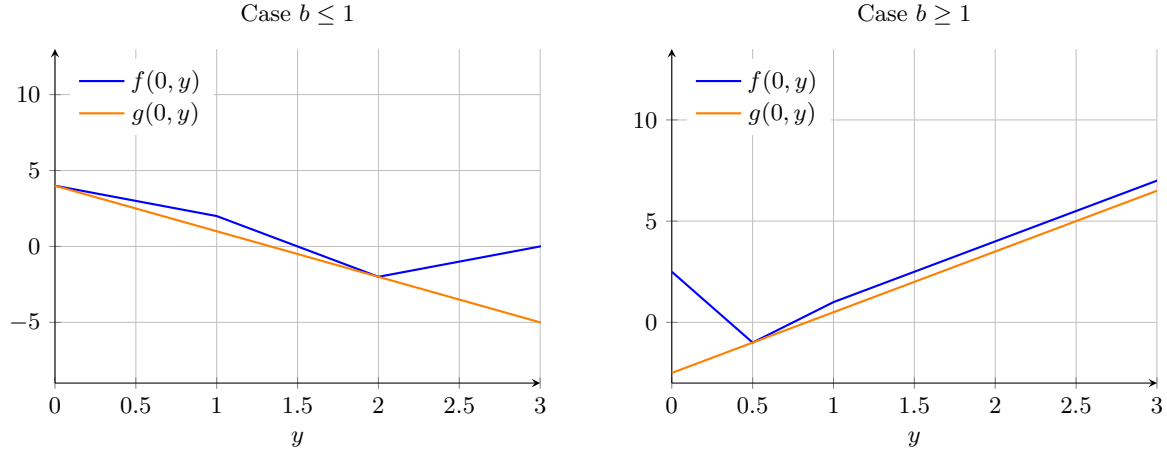       
As an auxiliary step, we first fix $x \ge 0$ and compute the convex hull of $f(x,\cdot)$ which we denote by $g(x,\cdot)$.
Depending on the value of $b$ we have
\[
\frac{d}{dz} f(x,y) = (l + x) \cdot 
\begin{cases}
-b(\lambda + \kappa) + 2 \kappa a, & y < \min(1,\frac1b), \\
b (\lambda - \kappa) + 2 \kappa a, & \frac1b < y < 1, \\
-b(\lambda + \kappa), & 1 < y < \frac1b, \\
b (\lambda - \kappa), & \max(1,\frac1b) < y.
\end{cases}
\]
If $b \le 1$ and $\lambda \ge \kappa$, the convex hull is given by
\begin{align*}
g(x,y) &= 
\begin{cases}
(1 - by) f(x,0) + by f(x,\frac1b), & y \le \frac1b \\
f(x,y), & y \ge \frac1b
\end{cases}\\
&= (l + x) \cdot
\begin{cases}
|1 - by| (\lambda - \kappa) + 2 by \kappa (a - 1), & y \le \frac1b \\
|1 - by| (\lambda - \kappa) + 2 \kappa (a - 1), & y \ge \frac1b
\end{cases}.
\end{align*}
Similarly, when $b \ge 1$ and $\lambda \ge \kappa$, we find
\begin{align*}
g(x,y) &=
\begin{cases}
f(x,y), & y \le \frac1b \\
                f(x,\frac1b) + (y - \frac1b) (l + x) b(\lambda - \kappa), & y \ge \frac1b
            \end{cases}\\
            &= (l + x) \cdot
            \begin{cases}
                 |1 - by| (\lambda - \kappa) + 2 \kappa y (a - b) & y \le \frac1b \\
                 |1 - by| (\lambda - \kappa) + 2 \kappa(a/b - 1) & y \ge \frac1b 
            \end{cases}
            .
        \end{align*}
        We observe that, in both cases, $g$ can be written as
        \[
            g(x,y) = (l + x) \Big( |1 - by| (\lambda - \kappa) + 2 \kappa \frac{\min(by,1)}{\max(b,1)} \Big( a - \max(b,1) \Big) \Big),
        \]
and $f^{\ast\ast}$ is also the convex hull of $g$.
To obtain a candidate for $f^{\ast\ast}$, we note that $g(w,\cdot)$ is V-shaped with ridge at $\frac1b$, $g(\cdot, \frac1b)$ is an affine function, and $g(0,\cdot)$ is affine when restricted to $[0,\frac1b]$ or $[\frac1b,\infty)$.
        
        Therefore, for $b \neq 1$, we consider the affine hyperplane $h$ coinciding with $g$ at the points $(0,1/b)$, $(1,1/b)$ and $(0,1)$, that is, 
        \[
            h(x,y) := 
            g(0,1/b) + x \big( g(1,1/b) - g(0,1/b) \big) + \frac{y - 1/b}{1 - 1/b} \big( g(0,1) - g(0,1/b) \big).
        \]
        The values of $g$ at these points are given by
        \begin{align*}
            g(0,1/b) &= 2 l \kappa \Big( \frac{a}{\max(1,b)} - 1 \Big),\\
            g(1,1/b) - g(0,1/b) &= 2 \kappa  \Big( \frac{a}{\max(1,b)} - 1 \Big),\\
            g(0,1) - g(0,1/b) &= l\Big( |1-b| (\lambda - \kappa) + \frac{2 \kappa}{\max(b,1)} \big( a - \max(b,1) \big) \big( \min(b,1) - 1 \big)  \Big).
        \end{align*}
        Substituting these values into the definition of $h$ yields that
        \begin{align*}
            h(x,y) 
            &= 
            \begin{cases}
                2\kappa( l + x) (a - 1) + l (1 - by) ( \lambda -  \kappa - 2\kappa (a - 1) ) & b < 1, \\
                2\kappa( l + x) (a / b - 1) + l (by - 1) (\lambda - \kappa) & b > 1.
            \end{cases}            
        \end{align*}
        In the next step, we show that $g \ge h$ on $\mathbb R^+ \times \mathbb R^+$.
        On the one hand, when $b < 1$ we compute
        \begin{align*}
            h(x,y) &= 2 \kappa (a - 1 ) ( l b y + x ) + l (1 - by) ( \lambda - \kappa ) \\
            &\le 2 k ( a - 1 ) ( l + x ) \min(1,by) + l | 1 - by| (\lambda - \kappa) \le g(x,y),
        \end{align*}
        using that $\lambda \ge \kappa$, $a \le 1$ and $x \ge 0$.
        On the other hand, when $b > 1$ we obtain
        \begin{align*}
            h(x,y) &\le (l + x)  \big( 2\kappa(a/b - 1) + |1 - by|(\lambda - \kappa) \big) \le g(x,y),
        \end{align*}
        using once more that $\lambda \ge \kappa$, $0 \le a \le b$, and $x \ge 0$.

        We have shown that $h$ lies tangentially to $g$ with $h(\cdot,\frac1b) = g(\cdot,\frac1b)$ and satisfying $h(0,\cdot) = g(0,\cdot)$ on $[0,\frac1b]$ when $b < 1$ and on $[\frac1b,\infty]$ when $b > 1$.
        Furthermore, due to the continuous dependence on $b$, we also find the value of $f^{\ast\ast}$ when $b = 1$, from where we deduce that
        \[
            f^{\ast\ast}(xl,1) = h(xl,1) = l \cdot
            \begin{cases}
                2\kappa(b + x)(a - 1) + |1 - b|(\lambda - \kappa) & b \le 1, \\
                2\kappa(1 + x)(a/b - 1) + |1 - b|(\lambda - \kappa) & b \ge 1.
            \end{cases}
        \]
        Finally, the assertion follows then by noting that $f^{\ast\ast}$ is decreasing in its first argument.
    \end{proof}

    \begin{corollary} \label{cor:ATTV}
        In the setting of Lemma \ref{lem:ATTV}, we have
        \begin{equation}
        \begin{split}
            &\inf_{ \substack{Y, \, \Delta W \ge 0, \\ \mathbb E[Y] = 1, \, \mathbb E[\Delta W] \le cl} } \mathbb E\left[ H(l,\Delta W,\lambda,\kappa,a,b,Y) \right]\\ 
            \ge
            &l \big( |1 - b|(\lambda - \kappa - 2 \kappa(c + 1)) + 2 \kappa (c + 1) (a - \min(b,1) \big).
            \end{split}
        \end{equation}
    \end{corollary}

    \begin{proof}
        Using the basic inequality
        \[
            a - 1 = a - \min(b,1) + \min(b,1) - 1 \ge
            a - \min(b,1) - | 1 - b|,
        \]
        we directly derive the claim from \eqref{eq:lem.ATTV}.
    \end{proof}

We are finally ready to prove the announced estimate. 
\begin{proof}[Proof of the estimate \eqref{eq:thm.weighted_main_inequality} in Theorem~\ref{thm:weighted_main_inequality}]
We start by considering the case where $\mu \ll \nu$. For $t = 1,\ldots,T$, we use the notation
\[
    Y_t \coloneqq  \frac{D^1_t}{D^2_t},~ A_t := \prod_{s = 1}^t \min \big(  Y_s,1 \big), ~
    B_t \coloneqq \frac{Z^1_t}{Z^2_t} = \prod_{s = 1}^t\frac{D^1_s}{D^2_s} =  \prod_{s = 1}^t Y_s,
\]
$W_t = w_{1:t}(X_{1:t}), ~ \Delta W_{t+1} = W_{t+1} - W_{t}$. Moreover, we recursively define the constants
\begin{align*}
\kappa_T &:= 1, \quad
\kappa_t := (c_{t+1} + 1)\kappa_{t+1}\quad \text{for $t=T-1, \ldots,1$,\quad and} \\
\lambda_1 &:= \kappa_1, \quad
\lambda_{t+1} := \lambda_t + \kappa_{t+1} + 2\kappa_{t+1}(c_{t+1}+1) - \kappa_{t}\quad \text{for $t=1,\ldots,T-1$},
\end{align*}
which in particular gives $\lambda_T = 1 + 2\sum_{t=1}^{T-1}\prod_{s=t+1}^{T}(1+c_s)$. Also, recall that
\begin{equation*}
    H(l,u,\lambda,\kappa,a,b,y) = (l + u) \left((\lambda - \kappa)|1-by| + 2\kappa \big( a \min(1,y) - \min(1,by) \big)\right).
\end{equation*}
Then, by applying Corollary~\ref{cor:ATTV} with
\[
l = W_{t-1},\quad u = \Delta W_T,\quad \lambda = \lambda_t,\quad \kappa = \kappa_t,\quad  a =A_{t-1},\quad b = B_{t-1},\quad y = Y_t,\quad c = c_t,
\]
for all $t = 1,\ldots,T-1$, we have
\begin{align}
\label{eq:thm:weighted_main_inequality.1}
\begin{split}
    &\quad~\E_\nu\Big[H(W_{t},\Delta W_{t+1},\lambda_{t+1},\kappa_{t+1},A_{t},B_{t},Y_{t+1})\Big|\calF_{t}\Big]\\
    &\geq W_{t} \big( (\lambda_{t} - \kappa_{t})|1 - B_{t}| + 2 \kappa_{t} (A_{t} - \min(B_{t},1) \big)\\
    &= (W_{t-1} + \Delta W_t) \big( (\lambda_{t} - \kappa_{t})|1 - B_{t-1}Y_{t}| + 2 \kappa_{t} (A_{t-1}\min(Y_t,1) - \min(B_{t-1}Y_{t},1) \big)
    \\
    &= H(W_{t-1},\Delta W_{t},\lambda_{t},\kappa_{t},A_{t-1},B_{t-1},Y_{t}).
    \end{split}
\end{align}
We combine this estimate with the expression of $\lambda \TV_w(\mu,\nu) - \ATV_w(\mu,\nu)$ obtained in \eqref{eq:lambdaTVwATVw}. By inductively applying \eqref{eq:thm:weighted_main_inequality.1} and the tower property of conditional expectation, it follows that
\begin{equation}
\label{eq:thm:weighted_main_inequality.2}
\begin{split}
    \lambda_T\TV_w(\mu,\nu) - \ATV_w(\mu,\nu)
    &=  \mathbb E_\nu\Big[ H(W_{T-1},\Delta W_{T},\lambda_{T},\kappa_{T},A_{T-1},B_{T-1},Y_{T})\Big]\\
    &\geq \cdots \\
    &\geq \mathbb E_\nu\Big[ W_{1} \big( (\lambda_{1} - \kappa_{1})|1 - B_{1}| + 2 \kappa_{1} (A_{1} - \min(B_{1},1) \big)\Big] = 0,
\end{split}
\end{equation}
where the last equality follows from $A_{1} = \min(B_{1},1)$ and $\kappa_1 = \lambda_1$.
        
Next, we prove the result in the general case, by approximating $\nu$ with $\nu^\epsilon$ such that $\mu \ll \nu^\epsilon$. We remain in the probabilistic setting introduced at the beginning of Section~\ref{sec:ATV_TV}, and introduce the stopping time $\tau := \inf \{ t \in \{1,\dots, T\} : D^2_t = 0 \}$, with the convention that the infimum over the empty set is defined as $T$.
        Fix $\epsilon \in (0,1)$ and define 
        \[
            D^{2,\epsilon}_t := \mathbbm 1_{\{ t < \tau \} } R_t^{2,\epsilon} D^2_t + \mathbbm 1_{\{ t = \tau \} } R_t^{1,\epsilon} D^1_t + \mathbbm 1_{\{t > \tau \} } D^1_t, \quad t =1,\ldots, T,
        \]
        where
        \[
            R_t^{1,\epsilon} \coloneqq \frac{\epsilon^2 \min(W_{t-1},1)}{\max(\mathbb E[W_t D^1_t \mathbbm 1_{\{D^2_t =0 \} } | \mathcal F_{t-1}], \epsilon)}, 
            R_t^{2,\epsilon} \coloneqq 1 - R_t^{1,\epsilon} \mathbb E[D^1_t\mathbbm 1_{\{D^2_t = 0\} } | \mathcal F_{t-1}],  t =1,\ldots, T.
        \]
        By this construction, we have $0 < R^{1,\epsilon}_t \le \epsilon$ and $1 \ge R^{2,\epsilon}_t \ge 1 - \epsilon^2$, so that $D_t^{2,\epsilon} \ge 0$.
        Moreover, $(D_t^{2,\epsilon})_{t = 1}^T$ is a conditional density process, that is, $D_t^{2,\epsilon}$ is $\mathcal F_t$-measurable and satisfies
        \begin{align*}
            \mathbb E\big[ D_t^{2,\epsilon} | \mathcal F_{t-1} \big] &=
            \mathbb E\big[ D^1_t |\mathcal F_{t-1} \big] \mathbbm 1_{ \{ t > \tau\} } +
            \mathbb E\big[ \mathbbm 1_{ \{D^2_t = 0 \} } R_t^{1,\epsilon} D^1_t + R_t^{2,\epsilon} D^2_t | \mathcal F_{t-1} \big] \mathbbm 1_{ \{ t \le \tau \} }
            \\
            &= \mathbbm 1_{ \{ t > \tau \} } + \mathbbm 1_{ \{ t \le \tau \} } \big( R_t^{1,\epsilon} \mathbb E[D^1_t \mathbbm 1_{\{ D^2_t = 0\}} | \mathcal F_{t-1}] + R_t^{2,\epsilon} \big) = 1.
        \end{align*}
       Now, we let $\nu^\epsilon$ be the law of the canonical process $X$ under $Z^{2,\epsilon}\sP$, with $Z^{2,\epsilon} = Z^{2,\epsilon}_T$, $Z^{2,\epsilon}_t = \prod_{s=1}^t D_s^{2,\epsilon}$, $t\leq T$, and set $w^\tau := w_{1:\tau}$.
        We claim that the pair $(\nu^\epsilon, w^\tau)$ satisfies Assumption~\ref{ass:conditional.moments}, with constants $(c_t^\epsilon)_{t = 1}^T$ given by $c_t^\epsilon = c_t + \epsilon^2$, and weighting sequence $(w_t^\tau)_{t = 1}^T$ given by $w_t^\tau(x_{1:t}) := w_{ t \wedge \tau }(x_{1:t\wedge \tau})$ for all $t\leq T$.\footnote{Note that as we are working here in the canonical setting and $\tau$ is a stopping time, $t \wedge \tau$ can be seen as a function of $x_{1:t}$.}
        To see this, we compute (again in the equivalent probabilistic setting where $W_t^\tau := w_t^\tau(X_1,\ldots,X_{t \wedge \tau})$)
        \begin{align*}
            &\quad~\mathbb E[D_t^{2,\epsilon}( W^\tau_t - W^\tau_{t-1} ) | \mathcal F_{t-1} ]\\
            &=
            \begin{cases}
                R_t^{1,\epsilon} \mathbb E[\mathbbm 1_{\{ D^2_t = 0 \} } D^1_t (W_t - W_{t-1}) | \mathcal F_{t-1}] + R_t^{2,\epsilon} \mathbb E[D^2_t(W_{t} - W_{t-1}) |\mathcal F_{t-1}], & t \le \tau, \\
                0,  & t > \tau.
            \end{cases}
        \end{align*}
        On $\{ \tau \ge t \}$ we have the bound
        \begin{align*}
            &\quad~R_t^{1,\epsilon} \mathbb E[\mathbbm 1_{\{ D^2_t = 0 \} } D^1_t (W_t - W_{t-1}) | \mathcal F_{t-1}] + R_t^{2,\epsilon} \mathbb E[D^2_t(W_{t} - W_{t-1}) |\mathcal F_{t-1}] \\
            &\le
            R_t^{1,\epsilon} \mathbb E[\mathbbm 1_{\{ D^2_t = 0 \} } D^1_t W_t|\mathcal F_{t-1}] + R_t^{2,\epsilon} c_t W_{t-1}
            \\
            &\le (\epsilon^2 + c_t) W_{t-1} = c_t^\epsilon W_{t-1},
        \end{align*}
        which proves the claim. 
        Since $\mu \ll \nu^\epsilon$, we can apply \eqref{eq:thm:weighted_main_inequality.2} with $\lambda^\epsilon := 1 + 2\sum_{t = 1}^{T-1} \prod_{s = t + 1}^T (1 + c_s^\epsilon)$, and get
        \begin{align*}
            \lambda^\epsilon \TV_{w^\tau}(\mu,\nu^\epsilon) - \ATV_{w^\tau}(\mu,\nu^\epsilon) \ge 0.
        \end{align*}
         Note that, if $T > \tau$, $\prod_{t = 1}^T \min(D^1_t,D^2_t) =  \min( Z^1,Z^2 ) = 0$. Thus 
         \begin{align*}
            &\quad~\lambda \TV_w(\mu,\nu) - \ATV_w(\mu,\nu)\\ 
            &=
            \mathbb E\Big[ W_T \Big( (\lambda - 1)|Z^1 - Z^2| + 2 \Big( \prod_{t = 1}^T \min(D^1_t,D^2_t) - \min( Z^1,Z^2 ) \Big) \Big) \Big] \\
            &=
            \mathbb E\Big[ (W_T - W^\tau_T) (\lambda - 1) |Z^1 - Z^2| + W^\tau_T \Big( (\lambda - 1)|Z^1 - Z^2|\\
            &\quad \quad + 2 \Big( \prod_{t = 1}^T \min(D^1_t,D^2_t) - \min( Z^1,Z^2 ) \Big) \Big] 
            \\
            &\ge \mathbb E\Big[ W^\tau_T \Big( (\lambda - 1)|Z^1 - Z^2| + 2 \Big( \prod_{t = 1}^T \min(D^1_t,D^2_t) - \min( Z^1,Z^2 ) \Big) \Big]  \\
            &= \lambda \TV_{w^\tau}(\mu,\nu) - \ATV_{w^\tau}(\mu,\nu).
        \end{align*}
        Since $\TV_{w^\tau}(\mu,\nu) = \lim_{\epsilon \to 0} \TV_{w^\tau}(\mu,\nu^\epsilon)$, $\ATV_{w^\tau}(\mu,\nu) = \lim_{\epsilon \to 0} \ATV_{w^\tau}(\mu,\nu^\epsilon)$, and $\lambda = \lim_{\epsilon \to 0} \lambda^\epsilon$, we conclude the proof of the estimate in \eqref{eq:thm.weighted_main_inequality}. Sharpness is shown in Subsection~\ref{subsec:sharpTVATV}.
    \end{proof}
    
We conclude this subsection by proving Theorem~\ref{thm:AWpTVp}, that follows from the estimates in Theorem~\ref{thm:weighted_main_inequality}.

\begin{proof}[Proof of Theorem~\ref{thm:AWpTVp}] The second inequalities are simply \eqref{eq:thm.ATVTV} and \eqref{eq:thm.ATVpTVp} in Theorem~\ref{thm:weighted_main_inequality}, so we only need to prove the first ones. By the definition of $\AW_p$ and $\ATV_p$, we have
\begin{equation}
\begin{split}
    \AW_p^{p}(\mu,\nu) & = \inf_{\pi \in \bccpl(\mu,\nu)}\int \mathbbm{1}_{\{x \neq y\}}(x,y) |x-y|_p^p \pi(dx,dy)\\
    &\leq \inf_{\pi \in \bccpl(\mu,\nu)}\int \mathbbm{1}_{\{x \neq y\}}(x,y) 2^p\big(|x|_p^p+|y|_p^p\big) \pi(dx,dy)\\
    &\leq 2^p \ATV_p^p(\mu,\nu).
\end{split}
\end{equation}
Now assume $\calK\subset \calX$ compact and $\mu,\nu \in \calP(\calK)$. Then 
\begin{equation}
\begin{split}
    \AW_p^{p}(\mu,\nu) \leq \inf_{\pi \in \bccpl(\mu,\nu)}\int \mathbbm{1}_{\{x \neq y\}}(x,y) \max_{x,y \in \calK}|x-y|_p^p \pi(dx,dy) \leq \mathrm{diam}(\calK)^p\ATV(\mu,\nu),
\end{split}
\end{equation}
which concludes the proof.
\end{proof}

\subsection{Sharpness of the constant in Theorem~\ref{thm:weighted_main_inequality}}
\label{subsec:sharpTVATV}
In this section, we present examples that illustrate both the sharpness of the bound established in Theorem~\ref{thm:weighted_main_inequality}, and the necessity of its dependence on the coefficients $c_t$.
\forestset{
  declare toks={elo}{font=\footnotesize, inner sep=2pt, midway, sloped}, 
  anchors/.style={anchor=#1,child anchor=#1,parent anchor=#1},
  dot/.style={tikz+={\fill (.child anchor) circle[radius=3pt];}},
  decision edge label/.style 2 args={
    edge label/.expanded={node[anchor=#1,\forestoption{elo}]{$#2$}}
  },
  decision/.style={if n=1
    {decision edge label={south}{#1}}
    {decision edge label={north}{#1}}
  },
  decision tree/.style={
    for tree={
      grow'=east,
      s sep=1em,
      l=13ex,
      if n children=0{anchors=west}{
        if n=1{anchors=south east}{anchors=north east}
      },
    },
    anchors=east,
    dot, for descendants=dot,
    delay={for descendants={split option={content}{;}{content,decision}}},
  },
}

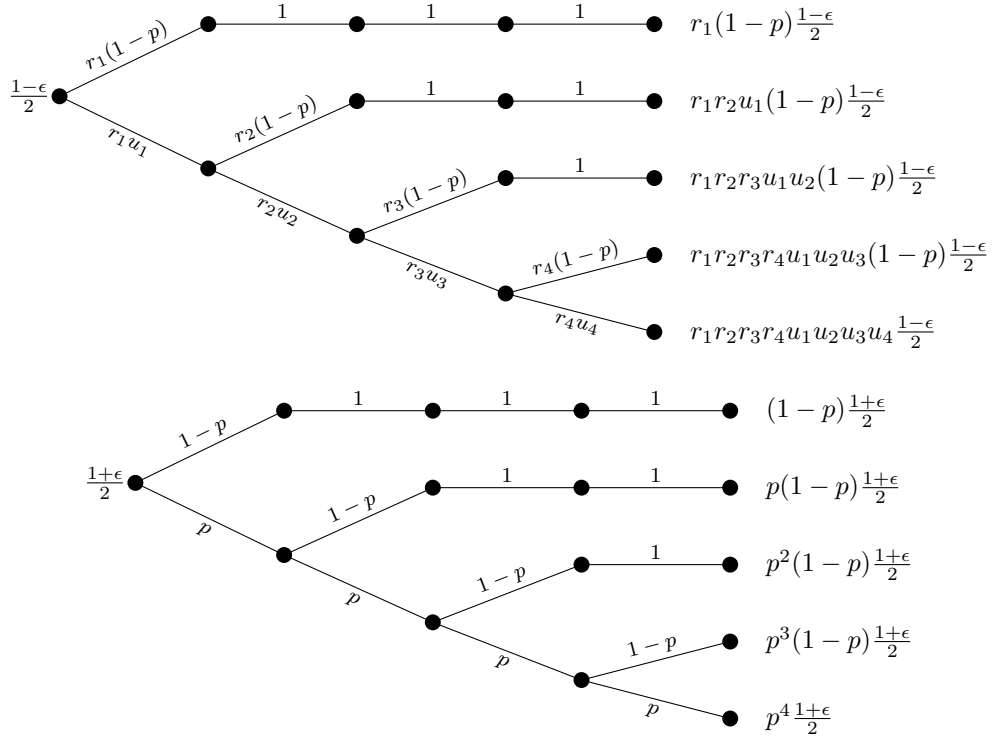
\begin{figure}
\centering
\begin{forest} decision tree
[$\frac{1 - \epsilon}{2}$
  [;r_1(1-p)[;1[;1[$\quad r_1(1-p)\frac{1 - \epsilon}{2}$ ;1]]]]
  [;r_1 u_1
    [;r_2(1-p)[;1[$\quad r_1 r_2 u_1(1-p)\frac{1 - \epsilon}{2}$ ;1]]]
    [;r_2 u_2
      [;r_3(1-p)[$\quad r_1 r_2 r_3 u_1 u_2 (1-p)\frac{1 - \epsilon}{2}$ ;1]]
      [;r_3 u_3 
        [$\quad r_1 r_2 r_3 r_4 u_1 u_2 u_3 (1-p)\frac{1 - \epsilon}{2}$ ;r_4(1-p)]
        [$\quad r_1 r_2 r_3 r_4 u_1 u_2 u_3 u_4 \frac{1 - \epsilon}{2}$ ;r_4 u_4]
      ]
    ]
  ]
]
\end{forest}

\vspace{1em}

\begin{forest} decision tree
[$\frac{1 + \epsilon}{2}$
  [;1-p[;1[;1[$\quad (1-p)\frac{1 + \epsilon}{2}$;1]]]]
  [;p
    [;1-p[;1[$\quad p(1-p)\frac{1 + \epsilon}{2}$;1]]]
    [;p
      [;1-p[$\quad p^2(1-p)\frac{1 + \epsilon}{2}$;1]]
      [;p 
        [$\quad p^3(1-p)\frac{1 + \epsilon}{2}$;1-p]
        [$\quad p^4\frac{1 + \epsilon}{2}$;p]
      ]
    ]
  ]
]
\end{forest}

\caption{Tree representation of the distributions $\gamma_1$ (top) and $\gamma_2$ (bottom) in Example~\ref{ex:boundissharp}, for $T=5$. 
Edge labels indicate transition probabilities between nodes.}
\label{fig:tree}
\end{figure}

\begin{example}[Sharpness of Constants in Theorem~\ref{thm:weighted_main_inequality}]
\label{ex:boundissharp}    
We construct two finite signed Borel measures $\gamma_1, \gamma_2$ on $\R^{T}$ supported on the tree structures illustrated in Figure~\ref{fig:tree}. Let $x^i \in \R^{T}$ denote the path corresponding to the $i$-th leaf from top to bottom, $i=1,\dots,T$. Although the specific values of $x^i_t$ are irrelevant for computing $\ATV$, $\TV$, $\ATV_w$, or $\TV_w$, we define them for clarity as
$x^i = \sum_{j=i+1}^T e^j$,
where $e^j$ is the $j$-th standard basis vector of $\R^T$; for instance,  $x^1 = (0,1,1,\dots,1)$, $x^2 = (0,0,1,\dots,1)$, and so on.
Thus, both $\gamma_1$ and $\gamma_2$ are supported on $\{x^i : i = 1,\dots, T\}$, and we construct them in a Markovian manner via transition kernels.
We want them to satisfy:
\begin{enumerate}[(i)]
    \item $\gamma_1(\{x^i\}) = \gamma_2(\{x^i\})$ for all $i \le T-1$,
    \item $(\gamma_1)_{x_{1:t}}(d x_{t+1}) \ne (\gamma_2)_{x_{1:t}}(d x_{t+1})$ for $x_{1:t} = 0 \in \R^t$ and $t = 1,\dots,T-1$.
\end{enumerate}
This setup is only possible if $\gamma_1(\R^T) \ne \gamma_2(\R^T)$. Hence, fix $\epsilon > 0$ and define their total masses as $\gamma_1(\R^T) = \frac{1 - \epsilon}{2}$ and $\gamma_2(\R^T) = \frac{1 + \epsilon}{2}$.
Let $p = \frac{2\epsilon}{1 + \epsilon}$, a choice we will justify shortly. The transition kernels of $\gamma_2$ are given by
\[
\begin{cases}
\gamma_2(X_{t+1} = 1\,|\,X_t = 0) = 1 - p,\\
\gamma_2(X_{t+1} = 0\,|\,X_t = 0) = p,\\
\gamma_2(X_{t+1} = 1\,|\,X_t = 1) = 1.
\end{cases}
\]

If $\gamma_1$ were to have the same transition kernel, we could not achieve $\gamma_1(\{x^i\}) = \gamma_2(\{x^i\})$ for all $i \le T-1$, due to the difference in total mass. To compensate, we scale the probability of transitioning upward at each node by a factor $r_t$, to be determined. That is, we define for $\gamma_1$
\[
\begin{cases}
\gamma_1(X_{t+1} = 1\,|\,X_t = 0) = r_t(1 - p),\\
\gamma_1(X_{t+1} = 0\,|\,X_t = 0) = 1 - r_t(1 - p),\\
\gamma_1(X_{t+1} = 1\,|\,X_t = 1) = 1.
\end{cases}
\]
We require $\gamma_1(\{x^i\}) = \gamma_2(\{x^i\})$ for all $i \le T-1$. To simplify notation, define $u_t \ge 0$ such that
\[
r_t u_t = 1 - r_t(1 - p), \quad \text{for all } t \le T - 1.
\]
Then the measures are given by
\begin{align*}
\gamma_2(\{x^i\}) &= \frac{1 + \epsilon}{2} (1 - p)p^{i-1}, \\
\gamma_1(\{x^i\}) &= \frac{1 - \epsilon}{2} (1 - p) r_1 \prod_{s=2}^i r_s u_{s-1}.
\end{align*}
Matching these two expressions yields
\[
r_t = \frac{1}{1 - d_t}\;\; \text{ and }\;\;  u_t = p - d_t, \quad \text{where} \quad d_t = \frac{2\epsilon}{(1 + \epsilon)p^{t-1}}, \quad t \le T - 1.
\]
To ensure validity, i.e., that the transition kernel of $\gamma_1$ is a probability kernel, we check the constraint $0 \le r_t(1-p) \le 1$, which reduces to
\[
d_t < 1 \quad \text{and} \quad p^t \ge \frac{2\epsilon}{1 + \epsilon}.
\]
This is satisfied when $p = \frac{2\epsilon}{1 + \epsilon}$ and $\epsilon < 1/2$, confirming our choice. In summary, we have constructed $\gamma_1$ and $\gamma_2$ such that
\[
\gamma_1(\{x^i\}) = \gamma_2(\{x^i\}) \quad \text{for all } i \le T-1, \quad \text{while} \quad (\gamma_1)_{x_t} \ne (\gamma_2)_{x_t} \text{ for all } t \le T-1.
\]

We now use $\gamma_1$ and $\gamma_2$ to define probability measures $\mu_\epsilon, \nu_\epsilon \in \calP(\R^T)$. For any $j \in \{-1, 1\}$ and $x \in \R^T$, define the shift $x^j := x + j \sum_{t=1}^T e^t$. Then define the shifted measures $\gamma_k^j := (x \mapsto x^j)_{\#} \gamma_k$ for $k = 1,2$.
Finally, set
\[
\mu_\epsilon := \gamma_1^1 + \gamma_2^{-1}, \qquad \nu_\epsilon := \gamma_2^1 + \gamma_1^{-1}.
\]
In other words, $\nu_\epsilon$ is obtained by swapping the initial positions of the two components in $\mu_\epsilon$.
For simplicity, we denote $\mu = \mu_\epsilon$ and $\nu = \nu_\epsilon$ in the following, and write $x^{i,j} := (x^i)^j$ for all $i \le T$ and $j \in \{-1, 1\}$.
\medskip

\noindent\textbf{The bound in \eqref{eq:thm.ATVTV} is sharp.}
We first compute the total variation
\begin{align*}
    \TV(\mu, \nu) &= \sum_{j\in \{1,-1\}}\sum_{i=1}^{T} |\mu(\{x^{i,j}\}) - \nu(\{x^{i,j}\})| \\
    &= 2 \sum_{i=1}^{T} |\gamma_1(\{x^i\}) - \gamma_2(\{x^i\})| 
    = 2 |\gamma_1(\{x^T\}) - \gamma_2(\{x^T\})| = 2\epsilon,
\end{align*}
where the second equality follows by symmetry, and the third one from the fact that $\gamma_1(\{x^i\}) = \gamma_2(\{x^i\})$ for all $i \leq T - 1$.

Next, we compute the adapted total variation. By \eqref{eq:lem.ATV.charlight.0}, we find
\begin{equation*}
    \begin{split}
        \ATV(\mu, \nu) 
        &= \sum_{j\in \{1,-1\}} \sum_{i=1}^{T} \bigg(\mu(\{x^{i,j}\}) + \nu(\{x^{i,j}\}) 
        - 2\Big( 
        \mu_1 \wedge \nu_1 \prod_{t = 1}^{T-1} \mu_{x_{1:t}} \wedge \nu_{y_{1:t}}
        \Big)(\{x^{i,j}\})\bigg) \\
        &= 2 \sum_{i=1}^{T} \bigg(\gamma_1(\{x^i\}) + \gamma_2(\{x^i\}) 
        - 2\Big( 
        (\gamma_1)_1 \wedge (\gamma_2)_1 \prod_{t = 1}^{T-1} 
        (\gamma_1)_{x_{1:t}} \wedge (\gamma_2)_{y_{1:t}}
        \Big)(\{x^i\})\bigg).
    \end{split}
\end{equation*}
As we have, for all $t \leq T - 1$, that $r_t(1 - p) > (1 - p)$ and $u_t r_t < p$, the following expression holds:
\begin{equation*}
    \Big( 
    (\gamma_1)_1 \wedge (\gamma_2)_1 \prod_{t = 1}^{T-1} 
    (\gamma_1)_{x_{1:t}} \wedge (\gamma_2)_{y_{1:t}}
    \Big)(\{x^i\}) = 
    \begin{cases}
        \frac{1 - \epsilon}{2} \prod_{t=1}^{i-1} u_t r_t (1 - p), & i \leq T - 1, \\
        \frac{1 - \epsilon}{2} \prod_{t=1}^{T - 1} u_t r_t, & i = T.
    \end{cases}
\end{equation*}
Recall that
\begin{equation*}
    \gamma_1(\{x^i\}) = 
    \begin{cases}
        \frac{1 - \epsilon}{2} (1 - p) r_1 \prod_{t=2}^{i} r_t u_{t-1} 
        = \gamma_2(\{x^i\}), & i \leq T - 1, \\
        \frac{1 - \epsilon}{2} \prod_{t=1}^{T - 1} r_t u_t 
        = \gamma_2(\{x^i\}) - \epsilon, & i = T,
    \end{cases}
\end{equation*}
and 
\begin{equation*}
    \gamma_2(\{x^i\}) =
    \begin{cases}
        \frac{1 + \epsilon}{2} (1 - p) p^{i - 1}, & i \leq T - 1, \\
        \frac{1 + \epsilon}{2} p^{T - 1}, & i = T.
    \end{cases}
\end{equation*}
Plugging this in, we obtain
\begin{equation*}
    \begin{split}
        \ATV(\mu, \nu) 
        &= 2 \sum_{i=1}^{T - 1} \left( \gamma_1(\{x^i\}) + \gamma_2(\{x^i\}) 
        - (1 - \epsilon) \prod_{t=1}^{i-1} u_t r_t (1 - p) \right)\\
        &~\quad 
        + 2 \left( \gamma_1(\{x^T\}) + \gamma_2(\{x^T\}) 
        - (1 - \epsilon) \prod_{t=1}^{T - 1} u_t r_t \right) \\
        &= 2 \sum_{i=1}^{T - 1} \left( 2\gamma_2(\{x^i\}) - \frac{2\gamma_2(\{x^i\})}{r_i} \right) 
        + 2 \left( 2\gamma_2(\{x^T\}) - \epsilon - 2(\gamma_2(\{x^T\}) - \epsilon) \right) \\
        &= 2 \sum_{i=1}^{T - 1} 2\gamma_2(\{x^i\}) d_i + 2\epsilon\\
        &= 2 \sum_{i=1}^{T - 1} 2 \cdot \frac{1 + \epsilon}{2}(1 - p)p^{i - 1} 
        \cdot \frac{2\epsilon}{(1 + \epsilon)p^{i - 1}} + 2\epsilon\\
        &= (2T - 2)(1 - p) \cdot 2\epsilon + 2\epsilon.
    \end{split}
\end{equation*}
Therefore, we have
\[
\lim_{\epsilon \to 0} \frac{\ATV(\mu_\epsilon, \nu_\epsilon)}{\TV(\mu_\epsilon, \nu_\epsilon)} 
= \lim_{\epsilon \to 0} \left( (2T - 2)\Big(1 - \frac{2\epsilon}{1+\epsilon}\Big) + 1 \right) 
= 2T - 1,
\]
which proves that the bound in \eqref{eq:thm.ATVTV} is sharp.
\medskip

\noindent\textbf{The bound in \eqref{eq:thm.weighted_main_inequality} is sharp.}
Let $(c_t)_{t = 2}^T \in \mathbb{R}_{\geq 0}^{T-1}$. Define weights $w: \mathbb{R}^T \to [0, \infty)$ and $w_t: \mathbb{R}^t \to [0, \infty)$ for $t \leq T$, by setting $w_T = w$ and, for all  $i \leq T,\ j \in \{-1, 1\}$,
\[
w_t(x^{i,j}_{1:t}) = 
\begin{cases}
    \prod_{s=i+1}^t (1 + c_s), & t > i\\
    1, & t \leq i
\end{cases}.
\]
By construction, $w_t(x^{i,j}_{1:t}) \leq w_{t+1}(x^{i,j}_{1:t+1})$ for all $t \leq T - 1$. Moreover,
\begin{equation*}
    \int w_{t+1}(x^{i,j}_{1:t}, \tilde{x}_{t+1}) \, \nu_{x^{i,j}_{1:t}}(d\tilde{x}_{t+1}) 
    \leq (1 + c_{t+1}) w_t(x^{i,j}_{1:t}).
\end{equation*}
Therefore, $(\nu_\epsilon, w)$ satisfies the conditional weighting Assumption~\ref{ass:conditional.moments} with $(c_t)_{t = 2}^T$.

We now compute the weighted total variation:
\begin{equation*}
    \TV_w(\mu, \nu) 
    = \sum_{j\in \{1,-1\}}\sum_{i=1}^{T} w(x^{i,j})|\mu(\{x^{i,j}\}) - \nu(\{x^{i,j}\})| 
    = 2 w(x^T) |\gamma_1(\{x^T\}) - \gamma_2(\{x^T\})| 
    = 2\epsilon.
\end{equation*}

The weighted adapted total variation is then:
\begin{equation*}
    \begin{split}
        \ATV_w(\mu, \nu) 
        &= 2 \sum_{i=1}^{T - 1} w(x^i) \left( \gamma_1(\{x^i\}) + \gamma_2(\{x^i\}) 
        - (1 - \epsilon) \prod_{t=1}^{i - 1} u_t r_t (1 - p) \right) \\
        &\quad + 2 w(x^T) \left( \gamma_1(\{x^T\}) + \gamma_2(\{x^T\}) 
        - (1 - \epsilon) \prod_{t=1}^{T - 1} u_t r_t \right) \\
        &= 2 \sum_{i=1}^{T - 1} 2\gamma_2(\{x^i\}) d_i w(x^i) + 2\epsilon \\
        &= 2\epsilon \sum_{i=1}^{T - 1} 2(1 - p) w(x^i) + 2\epsilon 
        = 2\epsilon(1 - p) \sum_{i=1}^{T - 1} 2\prod_{s = i + 1}^T (1 + c_s) + 2\epsilon.
    \end{split}
\end{equation*}

Therefore,
\[
\lim_{\epsilon \to 0} \frac{\ATV_w(\mu_\epsilon, \nu_\epsilon)}{\TV_w(\mu_\epsilon, \nu_\epsilon)} 
= \lim_{\epsilon \to 0} \left( 2(1 - p) \sum_{i=1}^{T - 1} \prod_{s = i + 1}^T (1 + c_s) + 1 \right)
= 2 \sum_{t=1}^{T - 1} \prod_{s = t + 1}^T (1 + c_s) + 1,
\]
establishing that the bound in \eqref{eq:thm.weighted_main_inequality} is sharp.
\end{example}

\begin{example}[The dependence on $c_t$ is necessary]
\label{ex:ct}
    For all $\epsilon \in (0,1)$, let $\mu^{\epsilon} = \epsilon(1-\epsilon)\delta_{(1,\frac{1}{\epsilon})} + \epsilon^{2}\delta_{(1,0)} + (1-\epsilon)\delta_{(0,0)}$ and $\nu^{\epsilon} = \epsilon(1-\epsilon)\delta_{(1,\frac{1}{\epsilon})} + (1-\epsilon + \epsilon^2)\delta_{(0,0)}$; see Figure~\ref{fig:ex_mu_nu_epsilon} for visualization.
     \begin{figure}
    \centering
    \begin{subfigure}{.4\textwidth}
      \centering
      \includegraphics[width=0.8\linewidth]{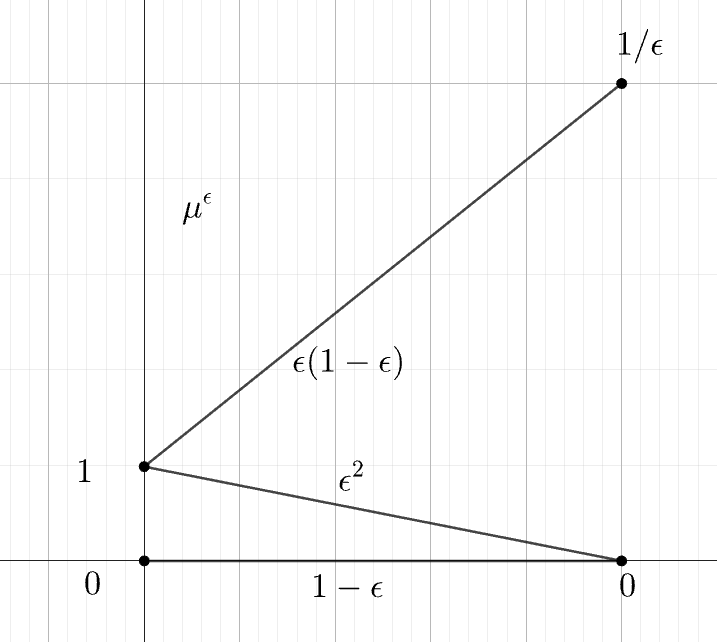}
      \caption{$\mu^\epsilon$}
    \end{subfigure}
    \begin{subfigure}{.4\textwidth}
      \centering
      \includegraphics[width=0.8\linewidth]{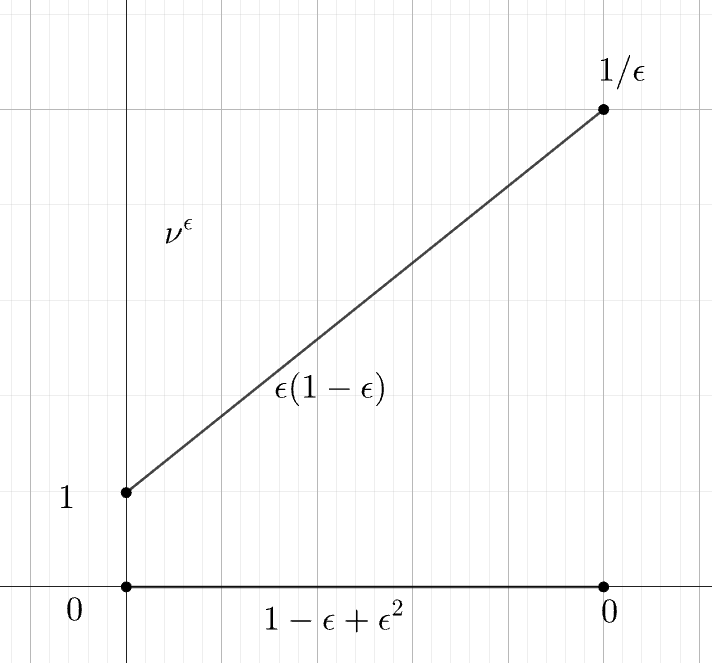}
      \caption{$\nu^\epsilon$}
    \end{subfigure}
    \caption{Visualization of $\mu^\epsilon$ and $\nu^\epsilon$ in Example~\ref{ex:ct}.}
    \label{fig:ex_mu_nu_epsilon}
\end{figure}
    Note that, for all $\epsilon \in (0,1)$, $\int \Vert x \Vert \mu^{\epsilon}(dx) \leq 2$ and $\int \Vert x \Vert \nu^{\epsilon}(dx) \leq 2$. Then, an easy computation gives $\TV_1(\mu^\epsilon, \nu^\epsilon) = 2\epsilon^2$ and $\ATV_1(\mu^\epsilon, \nu^\epsilon) = 2 + \epsilon - \epsilon^2$. Thus we have
    \begin{equation*}
        \lim_{\epsilon \to 0}\frac{\ATV_1(\mu^\epsilon, \nu^\epsilon)}{\TV_1(\mu^\epsilon, \nu^\epsilon)} = \lim_{\epsilon \to 0} \frac{2 + \epsilon - \epsilon^2}{2\epsilon^2} = +\infty,
    \end{equation*}
    which implies that there exists no constant $C >0$ independent of $(c_t)_{t=2}^{T}$ defined in \eqref{eq:thm.ATVVp.constraint} such that $\ATV_{1} \leq C\TV_{1}$, even for integrable measures. While this example is given for $p=1$, it easily generalizes to any $p\geq 1$ by $x\mapsto x^{\frac{1}{p}}$ transformation of the state space. 
\end{example}

\section{Linking $\TV$-distance and $\AW$-distance with $\W$-distance}
\label{sec:AW_W}
This section is devoted to the proof of Theorem~\ref{thm:AWtoW}. To obtain estimates for $\TV_p$ in terms of $\W_p$, we closely follow the arguments in \cite[Propositions 4.1.5 and 4.1.6]{gine2021mathematical}, extending the results there from the one-dimensional setting to multiple dimensions, and from the total variation $\TV$ to its $p$-moment weighted analogue $\TV_p$.

\subsection{Kernel approximation}

For any $k$-th order kernel $K$, cf.\ Definition~\ref{def:kthorderkernel}, and any \( h > 0 \), we define the rescaled kernel by \( K_h(\cdot) := \frac{1}{h}K\left(\frac{\cdot}{h}\right) \). As a first step, we estimate the approximation error of a function \( f \) by its smoothed version \( K_h * f \). 
\begin{lemma}
\label{lem:Kf_f}
    Let $f \in W^{k,1}(\R^{dT})$, $k\in\N$ and $K$ be a $k$-th order kernel.
    Then, for all $h>0$,
    \[
        \| K_h\ast f - f \|_1 \le h^k C_{k,K} \|f \|_{k,1},
    \]
    where $C_{k,K}\coloneqq \sup_{|\alpha| \leq k}\frac{1}{\alpha!}\int |K(z)z^\alpha|dz $.
\end{lemma}

\begin{proof}
    First we assume that $f\in C_c^{\infty}(\R^{dT})$. Then we have that, for any $x,y \in \mathbb R^{dT}$,    
    \[
        f(y) = f(x)+
        \sum_{|\alpha| < k} \frac{D^\alpha f(x)}{\alpha!} (y - x)^\alpha 
        + R_k(y,x),
    \]
    where $R_k(y,x)$ is the reminder term and $\int_{B(x,|y-x|)} |R_k(y,x)| dy  = o(|y-x|^k)$.
    This yields that
    \begin{align*}
        (K_h\ast f - f)(x) &= \int K(z) (f(x - zh) - f(x)) \, dz\\
        &= \int K (z) \sum_{|\alpha| = k} \frac{k}{\alpha!} \int_0^1 (1 - t)^{k - 1} D^\alpha f(x - t z h) \, dt (- zh)^\alpha \, dz,
    \end{align*}
    where we replaced $f(x - zh) - f(x)$ with its Taylor expansion and used orthogonality of the kernel $K$ to get rid of all terms of order smaller than $k$.
    Next, we proceed to estimate the $L_1$-norm of $K_h(f) - f$, and obtain
    \begin{align*}
        \| K_h\ast f - f \|_1 &=  \int \Big| \int K(z) \sum_{|\alpha| = k} \frac{k}{\alpha!} \int_0^1 (1 - t)^{k - 1} D^\alpha f(x - t z h) \, dt \, (- zh)^\alpha \, dz \Big| \, dx 
        \\
        &\le
        \sum_{|\alpha| = k} \frac{k}{\alpha!} \int |K(z) z^\alpha| h^k \int_0^1 (1 - t)^{k - 1}  \left( \int \Big| D^\alpha f(x - t z h) \, dz \Big| \, dx \right)
        \, dt \, dz
        \\
        &\le
        \sum_{|\alpha| = k} \frac{k}{\alpha!} \int |K(z) z^\alpha| h^k \int_0^1 (1 - t)^{k - 1} \|D^\alpha f\|_1
        \, dt \, dz
        \\
        &= h^k \sum_{|\alpha| = k} \frac{\|D^\alpha f\|_1 }{\alpha!} \int |K(z) z^\alpha| \, dz
        \le h^k C_{k,K} \|f\|_{k,1},
    \end{align*}
    Lastly, since $C_c^{\infty}(\R^{dT})$ is dense in $W^{k,1}(\R^{dT})$, and $f\mapsto \|f\|_{k,1}$ as well as  $f\mapsto \| K_h\ast f - f \|_1$ are continuous w.r.t.\ $\Vert \cdot \Vert_{k,1}$, this yields the assertion.
\end{proof}

\begin{lemma}
\label{lem:Kfw_Kgw}
Let $\mu,\nu \in\calP(\R^{dT})$ have densities $f,g$, and let $w\colon \R^{dT} \to \R_{\geq 0}$ be measurable. Then, for all kernels \( K \colon \R^{dT} \to \R \) and $h>0$, we have 
\begin{equation}
\label{eq:lem:Kfw_Kgw.0}
    \Vert K_h\ast (wf - wg) \Vert_1 \leq \| K \|_1 \left|\int w d(\mu-\nu)\right| + \sup_{\substack{\mathrm{Lip}(\varphi) \leq \mathrm{Lip}(K) \\  \varphi(0) = 0}} \frac{1}{h}\int \varphi w\,d(\mu -\nu).
\end{equation}
In particular,
    \begin{enumerate}[(i)]      \item when $w\equiv 1$, 
    \begin{equation}
    \label{eq:lem:Kfw_Kgw.1}
    \Vert K_h\ast (f - g) \Vert_1 \leq \frac1h \mathrm{Lip}(K) \W_1(\mu,\nu).
    \end{equation}
    \item when $w(x) = 1+|x|_p^p$, $p, q \geq 1$, 
    \begin{equation}
    \label{eq:lem:Kfw_Kgw.2}
    \begin{split}
        \Vert K_h\ast (wf - wg) \Vert_1 &\leq C_1\W_q(\mu,\nu) +  \frac1h C_2\W_q(\mu,\nu),
    \end{split}
    \end{equation}
    where 
    \begin{equation}
    \label{eq:lem:Kfw_Kgw.const}
        C_1 \coloneqq \| K \|_1 p\left(M_{\frac{q(p-1)}{q-1}}^{\frac{q-1}{q}}(\mu) + M_{\frac{q(p-1)}{q-1}}^{\frac{q-1}{q}}(\nu)\right), C_2 \coloneqq \mathrm{Lip}(K) \left(1+ p M_{\frac{qp}{q-1}}^{\frac{q-1}{q}}(\mu) + (p+1)M_{\frac{qp}{q-1}}^{\frac{q-1}{q}}(\nu)\right).
    \end{equation}
    \end{enumerate}
\end{lemma}
\begin{proof}
By the dual representation of $\| \cdot \|_1$, we get
\begin{equation*}
    \begin{split}
    &\quad~\Vert K_h\ast (wf - wg) \Vert_1\\ 
    &= \int \left| \int (fw - gw)(x-z)K_h(z)dz\right|dx\\
    &= \sup_{|\phi|\leq 1} \int \int \phi(x)(fw - gw)(x-z)K_h(z)dz dx\\
    &= \sup_{|\phi|\leq 1} \int (\phi\ast K_h)(x)w(x)(\mu - \nu)(dx)\\
    &= \sup_{|\phi|\leq 1} (\phi\ast K_h)(0) \int w d(\mu - \nu) + \int \Big((\phi\ast K_h)(x) - (\phi\ast K_h)(0)\Big)w(x)(\mu - \nu)(dx) \\
    &\leq \| K \|_1 \left|\int w d(\mu-\nu)\right| + \sup_{\substack{\mathrm{Lip}(\varphi) \leq \mathrm{Lip}(K) \\  \varphi(0) = 0}} \frac{1}{h}\int \varphi w\,d(\mu -\nu).
    \end{split}
\end{equation*}
This proves \eqref{eq:lem:Kfw_Kgw.0}. Next, to prove (i), by taking $w(x) = 1$, $x\in\R^{dT}$ and applying \eqref{eq:lem:Kfw_Kgw.0}, we have 
\begin{equation*}
    \begin{split}
        \Vert K_h\ast (f - g) \Vert_1 &\leq \| K \|_1 \left|\int d(\mu-\nu)\right| + \sup_{\substack{\mathrm{Lip}(\varphi) \leq \mathrm{Lip}(K) \\  \varphi(0) = 0}} \frac{1}{h}\int \varphi\,d(\mu -\nu)\\
        &= \sup_{\substack{\mathrm{Lip}(\varphi) \leq \mathrm{Lip}(K) \\  \varphi(0) = 0}} \frac{1}{h}\int \varphi\,d(\mu -\nu) = \frac1h \mathrm{Lip}(K) \W_1(\mu,\nu),
    \end{split}
\end{equation*}
where the first equality follows from $\int d\mu = \int d\nu = 1$, and the second one by the dual representation of the Wasserstein-$1$ distance. This proves \eqref{eq:lem:Kfw_Kgw.1}. Finally, to prove (ii), by taking $w(x) = (1 + |x|^p)$, $x\in\R^{dT}$ and applying \eqref{eq:lem:Kfw_Kgw.0}, we have 
\begin{equation}
\label{eq:ex:Kfw_Kgw.2:1.5}
    \begin{split}
        \Vert K_h\ast (wf - wg) \Vert_1 &\leq \| K \|_1 \left|\int (|x|^p - |y|^p)\pi(dx,dy)\right|\\
        &\qquad + \sup_{\substack{\mathrm{Lip}(\varphi) \leq \mathrm{Lip}(K) \\  \varphi(0) = 0}} \frac1h \int \varphi(x)(1 + |x|^p) (\mu -\nu)(dx).
    \end{split}
\end{equation}
Let $\pi \in \cpl(\mu,\nu)$. For the first term in \eqref{eq:ex:Kfw_Kgw.2:1.5}, by H\"{o}lder's inequality, for all $q > 1$,
\begin{equation}
\label{eq:ex:Kfw_Kgw.2:2}
\begin{split}
    \int (|x|^p - |y|^p)\pi(dx,dy) &\leq \int p(|x|^{p-1} + |y|^{p-1})|x-y|\pi(dx,dy)\\
    &\leq p\left(M_{\frac{q(p-1)}{q-1}}^{\frac{q-1}{q}}(\mu) + M_{\frac{q(p-1)}{q-1}}^{\frac{q-1}{q}}(\nu)\right)\W_q(\mu,\nu).
\end{split}
\end{equation}
For the second term in \eqref{eq:ex:Kfw_Kgw.2:1.5}, for all $q > 1$, $\varphi \colon \R^{dT} \to \R$ measurable s.t.\ $\varphi(0) = 0$ and $\mathrm{Lip}(\varphi) \leq \mathrm{Lip}(K)$,
\begin{equation}
\label{eq:ex:Kfw_Kgw.2:3}
\begin{split}
    &\quad~\int \varphi(x)(1 + |x|^p) (\mu -\nu)(dx)\\
    &= \int \left( \left(\varphi(x) - \varphi(y)\right)(1+|x|^p) + \varphi(y)(|x|^p - |y|^p) \right)\pi(dx,dy)\\
    &\leq \int \left(\frac{\varphi(x) - \varphi(y)}{|x-y|}|x-y|(1+|x|^p) + \varphi(y)|x-y|p(|x|^{p-1} + |y|^{p-1}) \right)\pi(dx,dy)\\
    &\leq \mathrm{Lip}(K)\int |x-y|\left(1+|x|^p + p|x|^{p-1}|y| + p|y|^{p} \right)\pi(dx,dy)\\
    &\leq \mathrm{Lip}(K)\int |x-y|\left(1+p|x|^p + (p+1)|y|^{p} \right)\pi(dx,dy),\quad (p|x|^{p-1}|y| \leq (p-1)|x|^p + |y|^p)\\
    &\leq \mathrm{Lip}(K) \left(1+ p M_{\frac{qp}{q-1}}^{\frac{q-1}{q}}(\mu) + (p+1)M_{\frac{qp}{q-1}}^{\frac{q-1}{q}}(\nu)\right)\W_q(\mu,\nu).
\end{split}
\end{equation}
By plugging \eqref{eq:ex:Kfw_Kgw.2:2} and \eqref{eq:ex:Kfw_Kgw.2:3} into \eqref{eq:ex:Kfw_Kgw.2:1.5}, we conclude that 
\begin{equation*}
    \begin{split}
        \Vert K_h\ast (wf - wg) \Vert_1 &\leq \| K \|_1 p\left(M_{\frac{q(p-1)}{q-1}}^{\frac{q-1}{q}}(\mu) + M_{\frac{q(p-1)}{q-1}}^{\frac{q-1}{q}}(\nu)\right)\W_q(\mu,\nu)\\
        &\quad + \frac1h \mathrm{Lip}(K) \left(1+ p M_{\frac{qp}{q-1}}^{\frac{q-1}{q}}(\mu) + (p+1)M_{\frac{qp}{q-1}}^{\frac{q-1}{q}}(\nu)\right)\W_q(\mu,\nu).
    \end{split}
\end{equation*}
This proves \eqref{eq:lem:Kfw_Kgw.2} and completes the proof of Lemma~\ref{lem:Kfw_Kgw}.
\end{proof}

\subsection{Proof of Theorem~\ref{thm:AWtoW}}
\begin{proof}[Proof of the estimates in Theorem~\ref{thm:AWtoW}]
    Recall that
    \begin{equation}
    \label{eq:thm:AW1toW1:1}
        \TV_p^p(\mu,\nu) = \Vert f_p - g_p\Vert_1 \leq \Vert f_p - K_h\ast f_p\Vert_1 + \Vert K_h\ast f_p - K_h\ast g_p\Vert_1 + \Vert K_h\ast g_p - g_p\Vert_1.
    \end{equation}
    By Lemma~\ref{lem:Kf_f}, we have
    \begin{equation}
    \label{eq:thm:AW1toW1:2}
        \| K_h \ast f_p - f_p \|_1 \le h^k C_{k,K} \|f_p \|_{k,1}, \quad \| K_h \ast g_p - g_p \|_1 \le h^k C_{k,K}\|g_p \|_{k,1}.
    \end{equation}
    Moreover, by Lemma~\ref{lem:Kfw_Kgw}-(ii), for all $q > 1$ we have that  
    \begin{equation}
    \label{eq:thm:AW1toW1:3}
        \Vert K_h\ast f_p - K_h\ast g_p\Vert_1 \leq C_1\W_q(\mu,\nu) + \frac1h C_2\W_q(\mu,\nu),
    \end{equation}
    with $C_1, C_2$ defined in \eqref{eq:lem:Kfw_Kgw.const}.
    By plugging \eqref{eq:thm:AW1toW1:2} and 
    \eqref{eq:thm:AW1toW1:3} into \eqref{eq:thm:AW1toW1:1}, and setting $h = \left(\frac{C_{2}\W_q(\mu,\nu)}{C_{k,K} (\|f \|_{k,1} + \|g \|_{k,1})}\right)^{\frac{1}{k+1}}$, we get
    \begin{equation*}
        \begin{split}
            \TV_{p}^p(\mu,\nu) &\leq C_1\W_q(\mu,\nu) + 2\big(C_{k,K}^{\frac{1}{k}}C_2\big)^{\frac{k}{k+1}}\left(\|f_p \|_{k,1} + \|g_p \|_{k,1}\right)^{\frac{1}{k+1}}\W_q^{\frac{k}{k+1}}(\mu,\nu).
        \end{split}
    \end{equation*}
    Combining this with Theorem~\ref{thm:AWpTVp}, we establish \eqref{eq:thm:AWtoW.1}. Next, to prove Theorem~\ref{thm:AWtoW}-(i), recall that 
    \begin{equation}
    \label{eq:thm:AW1toW1.compact:1}
        \TV(\mu,\nu) = \Vert f - g\Vert_1 \leq \Vert f - K_h\ast f\Vert_1 + \Vert K_h\ast f - K_h\ast g\Vert_1 + \Vert K_h\ast g - g\Vert_1.
    \end{equation}
    By Lemma~\ref{lem:Kf_f}, we have
    \begin{equation}
    \label{eq:thm:AW1toW1.compact:2}
        \| K_h \ast f - f \|_1 \le h^k C_{k,K} \|f \|_{k,1}, \quad \| K_h \ast g - g \|_1 \le h^k C_{k,K} \|g \|_{k,1},
    \end{equation}
    and, by Lemma~\ref{lem:Kfw_Kgw}-(i), 
    \begin{equation}
    \label{eq:thm:AW1toW1.compact:3}   
    \| K_h\ast f - K_h\ast g\|_1 \leq \frac1h \mathrm{Lip}(K) \W_1(\mu,\nu).
    \end{equation}
    By plugging \eqref{eq:thm:AW1toW1.compact:2} and 
    \eqref{eq:thm:AW1toW1.compact:3} into \eqref{eq:thm:AW1toW1.compact:1} and setting $h = \left(\frac{\mathrm{Lip}(K)\W_1(\mu,\nu)}{C_{k,K} (\|f \|_{k,1} + \|g \|_{k,1})}\right)^{\frac{1}{k+1}}$, we get
    \begin{equation*}  
    \TV(\mu,\nu) \leq 2\big(C_{k,K}^{\frac{1}{k}}\mathrm{Lip}(K)\big)^{\frac{k}{k+1}} \left(\|f_p \|_{k,1} + \|g_p \|_{k,1}\right)^{\frac{1}{k+1}} \W_1^{\frac{k}{k+1}}(\mu,\nu).
    \end{equation*}
    Combining this with Theorem~\ref{thm:AWpTVp}, we establish \eqref{eq:thm:AWtoW.2}. Finally, to prove Theorem~\ref{thm:AWtoW}-(ii), notice that $K\ast (f - g) = f - g$, because $K$ is a $k$-th order kernel which satisfies that $\int_{\mathbb{R}^{dT}} K(x)\, x^{\alpha} \, dx = 0$ for all $\alpha \in \mathbb{N}_0^{dT}$ with $1 \leq |\alpha| < k$.
    Then, by Lemma~\ref{lem:Kfw_Kgw}-(i) with $h=1$, we have 
    \[
    \TV(\mu,\nu) = \Vert f - g \Vert_1 = \Vert K_h\ast (f - g) \Vert_1 \leq \frac1h \mathrm{Lip}(K) \W_1(\mu,\nu) = \mathrm{Lip}(K) \W_1(\mu,\nu).
    \]
    Combining this with Theorem~\ref{thm:AWpTVp}, we establish \eqref{eq:thm:AWtoW.3}. 
\end{proof}

\begin{example}[The order in Theorem~\ref{thm:AWtoW} is sharp for $\AW_1$]
\label{ex:orderissharp}

For $k\in \N$ and $\epsilon \in (0,\frac{1}{8})$, we define distributions $\mu$ and $\nu$  (depending on $k, \epsilon$) on $[0,1]^2$ with densities $p_\mu$ and $p_\nu$, where $p_\mu = 1$ and
\begin{multline*}
    p_{\nu}(x_1,x_2)\\
    = 1 + \epsilon^k \sin(\frac{\pi x_1}{\epsilon}) \bigg[\mathbbm{1}_{\{x_2 \leq 2\epsilon\}}\Big( \sin(\frac{\pi x_2}{\epsilon} - \frac{\pi}{2}) + 1\Big) - \mathbbm{1}_{\{x_2 \geq 1-2\epsilon\}}\Big( \sin\big(\frac{\pi (1-x_2)}{\epsilon} - \frac{\pi}{2}\big) + 1\Big)\bigg];
\end{multline*}
see visualization in Figure~\ref{fig:density}. First, we compute the Soblev norm. Since the up-to-$k$-th order partial derivatives of $p_\mu$ are $0$ and the up-to-$k$-th order partial derivatives of $p_\nu$ are bounded in $[-2,2]$, there exists $c>0$ such that, for all $\epsilon \in (0,\frac{1}{8})$, $\|p_\mu \|_{k,1} + \|p_\nu \|_{k,1}\leq c$. Next, we compute the distances. As for the Wasserstein distance, we observe that at most $\epsilon^k$ mass is moved $2\sqrt{2} \epsilon$ far, so that $\W_1(\mu,\nu) \leq 2\sqrt{2}\epsilon^{k+1}$. Concerning the adapted Wasserstein distance, we have $\W_1(\mu_{x_1}, \nu_{y_1})\geq (1-4\epsilon) \epsilon^k |\sin(\frac{\pi x_1}{\epsilon})|$ for all $x_1,y_1 \in [0,1]$, since we have to move at least $\epsilon^k |\sin(\frac{\pi y_1}{\epsilon})|$ mass at a distance of at least $(1-4\epsilon)$. Therefore, for all $\epsilon \in (0,\frac{1}{8})$, we get the inequality
\[
\AW_1(\mu,\nu) \geq \inf_{\pi_1\in\cpl(\mu_1,\nu_1)}\int \W_1(\mu_{x_1}, \nu_{y_1})d\pi_1 = \int_0^1 (1-4\epsilon)\epsilon^k |\sin(\frac{\pi y_1}{\epsilon})|dy_1 \geq \frac{1}{2}\epsilon^k \frac{2}{\pi} = \frac{\epsilon^k}{\pi}.
\]
Thus, we get $\lim_{\epsilon \to 0} \AW_1(\mu,\nu) = \lim_{\epsilon \to 0} \W_1(\mu,\nu) = 0$ and
\[
\limsup_{\epsilon \to 0 }\frac{\AW_1(\mu,\nu)}{\W_1^{\frac{k}{k+1}}(\mu,\nu)} \geq \frac{1}{\pi 8^{\frac{1}{4}}} > 0.
\]
This proves that the order in Theorem~\ref{thm:AWtoW} is sharp for $\AW_1$.
\begin{figure}
    \centering
    \includegraphics[width=0.5\linewidth]{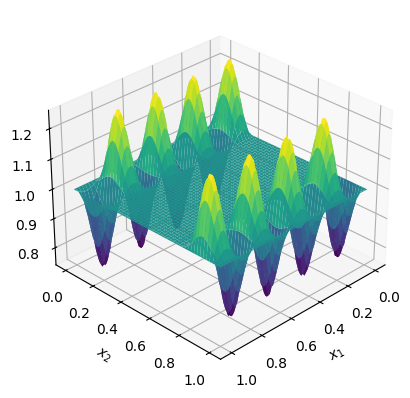}
    \caption{Density of $p_\nu$ in Example~\ref{ex:orderissharp}.}
    \label{fig:density}
\end{figure}
\end{example}

\section{Fast Convergence rate of $\AW_1$-distance}
\label{sec:fast}
This section is devoted to the proof of Theorem~\ref{thm:fast_rate_compact}. 

\begin{lemma}
\label{thm:negSobnorm}
Let $s, \gamma > 0$, $dT \geq 3$, $\mu \in \calP([0,1]^{dT})$ with density $f \in \mathcal{B}^s_{\infty,\infty}([0,1]^{dT})$. Assume further that $f \geq 1/\gamma$ over $[0,1]^{dT}$ for some $\gamma > 0$, and let $2^{J_n} \simeq n^{1/(2s + dT)}$. Then there exists a constant 
$C > 0$ depending on $\gamma,s$ such that
\begin{equation}
\label{eq:fast_rate_wass}
\mathbb{E}[\W_2(\hat \mu^n, \mu)] 
\leq C n^{-\frac{s+1}{2s+dT}}.
\end{equation}
\end{lemma}
\begin{proof}
    The estimator $\hat \mu^n$ defined in \eqref{eq:munhat} coincides with the one in \cite[Lemma~32]{manole2024plugin} when $\tilde f_n$ is non-negative. When $\tilde f_n$ takes negative values, the definition of $\hat \mu^n$ is different, but $\W_2(\hat \mu^n, \mu)$ remains bounded. Now, set $A_n = \{\tilde f_n \geq 0\}$. By \cite[Lemma~31]{manole2024plugin}, we have that $\sP(A_n^{\mathsf{c}}) \lesssim n^{-2}$. Since $\mu \in \calP([0,1]^{dT})$, it follows that on the event $A_n^{\mathsf{c}}$ we have $\W_2(\hat \mu_n, \mu) \leq 1 + \mathrm{diam}([0,1]^{dT}) = 1 + \sqrt{dT}$.
    Thus we have
    \begin{equation}
        \label{eq:fast_rate_wass_1}
        \mathbb{E}[\W_2(\hat \mu^n, \mu)] \leq \mathbb{E}[\W_2(\hat \mu^n, \mu)\mathbbm{1}_{A_n}] + \mathbb{E}[\W_2(\hat \mu^n, \mu)\mathbbm{1}_{A_n^{\mathsf{c}}}] \lesssim \mathbb{E}[\W_2(\hat \mu^n, \mu)\mathbbm{1}_{A_n}] + n^{-2}.
    \end{equation}
    Then, by \cite[Lemma~32]{manole2024plugin}, we have $\mathbb{E}[\W_2(\hat \mu^n, \mu)\mathbbm{1}_{A_n}] \lesssim n^{-\frac{s+1}{2s+dT}}$, which combined with \eqref{eq:fast_rate_wass_1}, completes the proof of \eqref{eq:fast_rate_wass}. 
\end{proof}

\begin{lemma}
\label{lem:norm_equiv}
    Let $r > 0$ and $r\notin \N$. Then there exists $C\geq 0$ such that, for any $g \colon [0,1]^{dT} \to \R$   differentiable up to order $\lfloor r \rfloor$ in the interior of $[0,1]^{dT}$, it holds that
    \[
    \Vert g \Vert_{W_{\lfloor r \rfloor,\infty}} \leq d^{\lfloor r \rfloor} \Vert g \Vert_{\mathcal{C}^{r}} \leq C \Vert g\Vert_{\mathcal{B}_{\infty, \infty}^{r}},
    \]
    where 
    \[
    \| g \|_{\calC^r} 
    \coloneqq \sum_{j=0}^{\lfloor r \rfloor} \sup_{|\alpha| = j} \| D^\alpha g \|_\infty
    + \sum_{|\alpha| = \lfloor r \rfloor} \sup_{\substack{x,y \in (0,1)^{dT}\\ x \neq y}}
    \frac{| D^\alpha g(x) - D^\alpha g(y) |}{\|x-y\|^{r - \lfloor r \rfloor}} .
    \]
\end{lemma}
\begin{proof}
    By Definition~\ref{def:sob}, we have 
    \[
    \Vert g \Vert_{W_{\lfloor r \rfloor,\infty}} = \sum_{|{\alpha}|\leq \lfloor r \rfloor} \Vert D^{{\alpha}}g \Vert_{\infty} \leq d^{\lfloor r \rfloor} \sum_{j=0}^{\lfloor r \rfloor} \sup_{|\alpha| = j} \| D^\alpha g \|_\infty \leq d^{\lfloor r \rfloor} \Vert g \Vert_{\mathcal{C}^{r}},
    \]
    so the first inequality in the statement holds. 
    Further, by \cite[Lemma~27]{manole2024plugin}, since $r\notin \N$, the $\Vert \cdot \Vert_{\mathcal{C}^{r}}$-norm and the $\Vert \cdot\Vert_{\mathcal{B}_{\infty, \infty}^{r}}$-norm are equivalent, so we get the second inequality, hence completing the proof.
\end{proof}

\begin{lemma}
\label{lem:Sobnorm_emp_density}
    Assume the setting of Theorem~\ref{thm:fast_rate_compact}. Then there exists $C >0$ such that
    \[
    \E\big[\Vert \tilde f_n \Vert_{W_{k,1}}\big] \leq  C \quad\text{for all $n \in \N$.}
    \]
\end{lemma}
\begin{proof}
    Recall that $f$ admits the wavelet expansion
    \[
    f = \sum_{\zeta \in \Phi} \beta_\zeta \zeta + \sum_{j=j_0}^{\infty}\sum_{\xi \in \Psi_j}\beta_\xi \xi,
    \]
    and that
    \[
    \tilde f_n = \sum_{\zeta \in \Phi}\hat \beta_\zeta^n \zeta + \sum_{j=j_0}^{J_n} \sum_{\xi \in \Psi_{j}} \hat\beta^n_\xi \xi.
    \]
    By \cite[Lemma~30]{manole2024plugin}, there exist constants $v, b > 0$ depending only on the choice of wavelet system such that, for any $J_n \geq j_0$, and all $u > 0$,
    \begin{equation}
    \label{eq:estimate_beta_1}
        \sP\Bigg( \sup_{\zeta \in \Phi} |\hat\beta_\zeta^n - \beta_\zeta| \geq u \Bigg) \lesssim \exp\left\{ - \frac{n u^2}{b} \right\},
    \end{equation}
    \begin{equation}
    \label{eq:estimate_beta_2}
    \sP\Bigg( \sup_{\xi \in \Psi_j} |\hat\beta_\xi^n - \beta_\xi| \geq u \Bigg) 
    \lesssim 2^{dTj/2} \exp\left\{ - \frac{n u^2}{v + b 2^{dTj/2} u} \right\}, 
    \quad j_0 \leq j \leq J_n.
    \end{equation}
    Let $r \in (k,s)$ and $r\notin \N$. A union bound combined with
    \eqref{eq:estimate_beta_2} leads to
    \begin{equation*}
    \sP\Bigg( \sup_{j_0 \leq j \leq J_n} \sup_{\xi \in \Psi_j} 
    \big|\hat\beta_\xi^n - \beta_\xi \big| \geq u \Bigg) 
    \lesssim J_n 2^{dT J_n/2} \exp\Bigg\{ - \frac{n u^2}{v + b2^{dT J_n /2} u} \Bigg\},
    \end{equation*}
    whence, since $2^{J_n} \simeq n^{\frac{1}{2s+dT}}$, 
    \begin{equation}
    \label{eq:beta_estimate}
    \begin{split}
    \sP\Bigg( 
    2^{\frac{J_n (2r+dT)}{2}} 
    \sup_{j_0 \leq j \leq J_n} \sup_{\xi \in \Psi_j}|\hat\beta_\xi^n - \beta_\xi| \geq u \Bigg)
    &\lesssim J_n 2^{dT J_n/2} 
    \exp\Bigg\{ 
    -\frac{n u^2 2^{-J_n (2r+dT)}}{v + b 2^{- J_n r} u} 
    \Bigg\}.
    \end{split}
    \end{equation}
    We further define
    \[
    f_n = \sum_{\zeta \in \Phi} \beta_\zeta \zeta + \sum_{j=j_0}^{J_n}\sum_{\xi \in \Psi_j}\beta_\xi \xi.
    \]
    Combining \eqref{eq:estimate_beta_1}, \eqref{eq:estimate_beta_2}, \eqref{eq:beta_estimate}, and using the definition of $ \| \cdot \|_{\calB^{r}_{\infty, \infty}}$, we have
    \begin{align*}
    &\sP\Big( \| \tilde f_n - f_n \|_{\calB^{r}_{\infty, \infty}} \geq u \Big)\\
    \leq~& 
    \sP\Big( \sup_{\zeta \in \Phi} |\hat\beta_\zeta^n - \beta_\zeta| \geq u/2 \Big) 
    + 
    \sP\Bigg( 
    2^{\frac{J_n (2r+dT)}{2}} 
    \sup_{j_0 \leq j \leq J_n} 
    \sup_{\xi \in \Psi_j} |\hat\beta_\xi^n - \beta_\xi| \geq u/2 
    \Bigg) \\
    \lesssim~& J_n 2^{dT J_n/2} 
    \exp\Bigg\{ 
    -\frac{n u^2 2^{-J_n (2r+dT)}}{4v + 2b 2^{- J_n r} u} 
    \Bigg\}.
    \end{align*}
    Therefore,
    \begin{align*}
    \E\Big[ \| \tilde f_n - f_n \|_{\mathcal{B}^{r}_{\infty, \infty}} \Big] &\leq  \sum_{u=1}^{\infty} u \sP\Big( \| \tilde f_n - f_n \|_{\mathcal{B}^{r}_{\infty, \infty}} \geq u \Big)\\
    &\lesssim J_n 2^{dT J_n/2} \sum_{u=1}^{\infty}u 
    \exp\Bigg\{ 
    -\frac{n u^2 2^{-J_n (2r+dT)}}{4v + 2b 2^{- J_n r} u} 
    \Bigg\}.
    \end{align*}
    Setting $A_n \coloneqq n 2^{-J_n (2r+dT)} \simeq n^{\frac{2(s-r)}{2s+dT}}$ and $B_n \coloneqq 2b 2^{-J_n r} \simeq n^{\frac{-r}{2s+dT}}$, we get
    \begin{equation}
        \label{eq:sum_u}
        \E\Big[ \| \tilde f_n - f_n \|_{\mathcal{B}^{r}_{\infty, \infty}} \Big] \lesssim J_n 2^{dTJ_n/2} \sum_{u=1}^\infty \exp\bigg\{\frac{-A_n u^2}{4v + B_n u}\bigg\}.
    \end{equation}
    Note that, for all $u \geq \frac{4v}{B_n}$, we have $\frac{-A_n u^2}{4v + B_n u} \leq \frac{-A_n u}{2B_n}$. Let $u_0 = \lfloor \frac{4v}{B_n} \rfloor$. We estimate 
    \begin{equation}
    \label{eq:sum_u_split}
        \sum_{u=1}^{\infty}u 
    \exp\bigg\{\frac{-A_n u^2}{4v + B_n  u}\bigg\} \leq \sum_{u=1}^{u_0}u 
    \exp\bigg\{\frac{-A_n u^2}{4v + B_n  u}\bigg\} + \sum_{u=u_0 + 1}^{\infty}u 
    \exp\bigg\{\frac{-A_n u}{2 B_n}\bigg\}.
    \end{equation}
    By denoting $\alpha_n \coloneqq \frac{A_n}{2 B_n} \simeq n^{\frac{2s-r}{2s+dT}}$, we have 
    \begin{equation}
    \label{eq:sum_u_part2}
        \sum_{u=u_0+1}^{\infty}u 
    \exp\bigg\{\frac{-A_n u}{2B_n}\bigg\} \leq \int_1^\infty u e^{-\alpha_n u} \, du = e^{-\alpha_n}\Big(\frac{1}{\alpha_n} + \frac{1}{\alpha_n^2}\Big) \lesssim 1.
    \end{equation}
    Also note that
    \begin{equation}
        \label{eq:sum_u_part1}
        \sum_{u=1}^{u_0}u 
    \exp\bigg\{\frac{-A_n u^2}{4v + B_n  u}\bigg\} \lesssim u_0^2 \exp\bigg\{\frac{-A_n}{4v + 4v}\bigg\} \lesssim n^{\frac{2r}{2s+dT}}\exp\bigg\{-\frac{n^{\frac{2(s-r)}{2s+dT}}}{8v}\bigg\} \lesssim 1.
    \end{equation}
    Combining \eqref{eq:sum_u}, \eqref{eq:sum_u_split}, \eqref{eq:sum_u_part1}, and \eqref{eq:sum_u_part2}, we obtain that there exists $C_0 > 0$ such that
    \[
    \E\Big[ \| \tilde f_n - f_n \|_{\calB^{r}_{\infty, \infty}} \Big] \leq C_0.
    \]
    Thus, we have
    \[
    \E\big[\Vert \tilde f_n \Vert_{\mathcal{B}_{\infty, \infty}^{r}}\big]  \leq \E\big[\Vert \tilde f_n - f_n \Vert_{\mathcal{B}_{\infty, \infty}^{r}}\big] + \Vert f_n \Vert_{\mathcal{B}_{\infty, \infty}^{r}} \leq C_0 + \Vert f\Vert_{\mathcal{B}_{\infty, \infty}^{r}}.
    \]
    By Lemma~\ref{lem:norm_equiv} and the fact that $r > k\in\N$, we have that, for any $g \colon [0,1]^{dT} \to \R$  differentiable up to order $\lfloor r \rfloor$ in the interior of $[0,1]^{dT}$, there exists $C_1\geq 0$ such that
    \[
    \Vert g \Vert_{W_{k,1}} \leq \Vert g \Vert_{W_{k,\infty}} \leq \Vert g \Vert_{W_{\lfloor r \rfloor,\infty}}\leq d^{\lfloor r \rfloor} \Vert g \Vert_{\mathcal{C}^{r}} \leq C_1 \Vert g\Vert_{\mathcal{B}_{\infty, \infty}^{r}}.
    \]
    Therefore, we conclude that there exists $C >0$ such that 
    \[
    \E\big[\Vert \tilde f_n \Vert_{W_{k,1}}\big] \leq \E\big[\Vert \tilde f_n \Vert_{\calC^r} \big] \leq C_1 (C_0 + \Vert f \Vert_{\calB_{\infty,\infty}^{r}}) \leq  C. \qedhere
    \]
\end{proof}

\begin{proof}[Proof of Theorem~\ref{thm:fast_rate_compact}]
First, by applying Theorem~\ref{thm:AWtoW}-(i) with $p=1$, we have
\begin{equation}
\label{eq:fast_rate.1}
\begin{split}
    \E\big[\AW_1(\mu,\hat \mu_n)\big] &\lesssim \E\bigg[\left(\|f \|_{k, 1} + \|\hat f_n \|_{k, 1}\right)^{\frac{1}{k+1}} \W_1^{\frac{k}{k+1}}(\mu,\hat \mu_n)\bigg]\\
    &\lesssim\E\big[\|f \|_{k, 1} + \|\hat f_n \|_{k, 1}\big]^{\frac{1}{k+1}} \E\big[\W_1(\mu,\hat \mu_n)\big]^{\frac{k}{k+1}},
\end{split}
\end{equation}
where the last inequality follows by H\"{o}lder's inequality. Let $r \in (k,s)$ and $r \notin \N$. By Lemma~\ref{lem:norm_equiv} and the fact that $r > k\in\N$, there exists $C_1 > 0$ such that 
\begin{equation}
    \label{eq:norm_bound_1}
    \|f \|_{W_{k, 1}} \leq \|f \|_{W_{k, \infty}} \leq C_1 \Vert f\Vert_{\mathcal{B}_{\infty, \infty}^{r}} < \infty.
\end{equation}
By Lemma~\ref{lem:Sobnorm_emp_density}, there exists $C_2 > 0$ such that for all $n \in \N$,
\begin{equation}
    \label{eq:norm_bound_2}
    \E\big[\|\tilde f_n \|_{W_{k, 1}}\big] < C_2.
\end{equation}
Moreover, we have
\begin{equation}
    \label{eq:norm_bound_3}
    \E\big[\Vert \varphi(x-X^{(1)})\Vert_{W_{k,1}}\big] = \Vert \varphi \Vert_{W_{k,1}} < \infty.
\end{equation}
Hence, combining \eqref{eq:norm_bound_1}, \eqref{eq:norm_bound_2}, and \eqref{eq:norm_bound_3}, there exists $C_3 > 0$ such that for all $n \in \N$,
\[
\E\big[\|f \|_{W_{k, 1}} + \|\hat f_n \|_{W_{k, 1}}\big] \leq \|f \|_{W_{k, 1}} + \E\big[\|\tilde f_n \|_{W_{k, 1}}\big] + \E\big[\Vert \varphi(x-X_1)\Vert_{W_{k,1}}\big] \leq C_3
\]
and
\begin{equation}
\label{eq:fast_rate.2}
    \E\big[\AW_1(\mu,\hat \mu_n)\big] \lesssim \E\big[\W_1(\mu,\hat \mu_n)\big]^{\frac{k}{k+1}} \leq \E\big[\W_2(\mu,\hat \mu_n)\big]^{\frac{k}{k+1}}.
\end{equation}
Finally, by combining \eqref{eq:fast_rate.1}, \eqref{eq:fast_rate.2} and Theorem~\ref{thm:negSobnorm}, we have
\begin{equation*}
    \E\big[\AW_1(\mu,\hat \mu_n)\big] \lesssim n^{-\frac{s+1}{2s+dT}\frac{k}{k+1}} \lesssim n^{-\frac{k+1}{2k+dT}\frac{k}{k+1}},
    \end{equation*}
which completes the proof.
\end{proof}

\appendix
\section{Boundary-Corrected Wavelets}
\label{app:wavelets}
In this section, we give an introduction of boundary-corrected wavelets for the sake of completeness. Our exposition closely follows that of Appendix~A.2 in \cite{manole2024plugin} and Section~4.3 in \cite{gine2021mathematical}. We also refer the reader to \cite{cohen1993wavelets} and references
therein for further details.

For an integer $N\geq 2$, we denote the compactly-supported $N$-th Daubechies scaling and wavelet functions by $\zeta_0, \xi_0 \in \mathcal{C}^r(\mathbb{R}^d)$, where $r = 0.18 (N - 1)$; see \cite[Theorem~4.2.10]{gine2021mathematical}. It is well-known that the $N$-th Daubechies wavelet system 
\[
\zeta_{0k} = \zeta_0(\,\cdot - k\,), 
\qquad 
\xi_{0jk} = 2^{\frac{j}{2}} \xi_0(2^j(\,\cdot\,) - k), 
\qquad 
j \ge 0,\, k \in \mathbb{Z},
\]
forms a basis of $L^2(\mathbb{R})$, with the property that 
$\{ \zeta_{0k} : k \in \mathbb{Z} \}$ 
spans all polynomials on $\mathbb{R}$ of degree at most $N - 1$. 
While this family may easily be periodized to obtain a basis for
$L^2([0,1])$, doing so may not accurately reflect the regularity 
of functions in $L^2([0,1])$ via the decay of their wavelet coefficients, near the boundaries of the 
interval. This consideration motivated 
the introduction of the so-called boundary-corrected wavelet system on $[0,1]$, which preserves the 
standard Daubechies scaling functions lying sufficiently far from the boundaries of the interval, 
and adds edge scaling functions such that their union continues to span all polynomials 
up to degree $N-1$ on $[0,1]$. In short, given a fixed integer $j_0 \ge \log_2 N$, the construction in \cite{cohen1993wavelets} leads to smooth scaling edge basis functions
\[
\zeta^{\text{left}}_{0j_0k} 
\quad \text{with support contained in } 
[0, (2N - 1)/2^{j_0}],
\]
\[
\zeta^{\text{right}}_{0j_0k} 
\quad \text{with support contained in } 
[1 - (2N - 1)/2^{j_0}, 1],
\]
which in turn can be used to define edge wavelet functions 
$\xi^{\text{left}}_{0j_0k}, \xi^{\text{right}}_{0j_0k}$, for $k = 0, \ldots, N-1$.
In this case, if one defines
\[
\zeta^{a}_{0jk} = 2^{\frac{j - j_0}{2}} \zeta^{a}_{0j_0k}(2^{j - j_0}(\cdot)), 
\quad 
\xi^{a}_{0jk} = 2^{\frac{j - j_0}{2}} \xi^{a}_{0j_0k}(2^{j - j_0}(\cdot)),
\]
for all $j \ge j_0$ and $a \in \{\text{left}, \text{right}\}$,
then the families
\begin{eqnarray*}
\Phi_0 
&=& \{ \zeta_{0j_0k} : 0 \le k \le 2^{j_0} - 1 \} \\
&=& \{ \zeta^{\text{left}}_{0j_0k}, \zeta^{\text{right}}_{0j_0k}, \zeta_{0m} : 0 \le k \le N - 1,\, N \le m \le 2^{j_0} - N - 1 \},
\end{eqnarray*}
\begin{eqnarray*}
\Psi_0 
&=& \{ \xi_{0jk} : 0 \le k \le 2^j - 1,\, j \ge j_0 \}\\
&=& \{ \xi^{\text{left}}_{0jk}, \xi^{\text{right}}_{0jk}, \xi_{0jm} : 0 \le k \le N - 1,\, N \le m \le 2^{j_0} - N - 1,\, j \ge j_0 \}
\end{eqnarray*}
form an orthonormal basis of $L^2([0,1])$, with the property that 
$\Phi$ spans all polynomials on $[0,1]$ of degree at most $N - 1$. We then define a tensor product wavelet basis of $L^2([0,1]^{dT})$ 
by setting, for all $j \ge j_0$ and all 
$\ell = (\ell_1, \ldots, \ell_{dT}) \in \{0,1\}^{dT} \setminus \{0\}$,
\[
\zeta_{j_0 k}(x) = \prod_{i=1}^{dT} \zeta_{0j_0k_i}(x_i), 
\qquad 
\xi_{j k \ell}(x) 
= \prod_{i:\ell_i = 0} \zeta_{0jk_i}(x_i) 
  \prod_{i:\ell_i = 1} \xi_{0jk_i}(x_i),
\qquad x \in [0,1]^{dT},
\]
where, in the definition of $\zeta_{j_0 k}$, 
the index $k = (k_1, \ldots, k_{dT})$ ranges over 
$\mathcal{K}(j_0) := \{1, \ldots, 2^{j_0} - 1\}^{dT}$, 
while in the definition of $\xi_{j k \ell}$, 
$k$ ranges over $\mathcal{K}(j)$. In this case, the wavelet system
\[
\Psi = \Phi \cup \bigcup_{j = j_0}^{\infty} \Psi_j,
\qquad
\Phi = \{ \zeta_{j_0k} : k \in \mathcal{K}(j_0) \}, 
\qquad
\Psi_j = \{ \xi_{j k \ell} : k \in \mathcal{K}(j) \}, 
\quad j \ge j_0,
\]
forms a basis of $L^2([0,1]^{dT})$. 

\printbibliography

\end{document}